\documentclass[twoside,11pt]{article}
\usepackage[hmargin=1in,vmargin=1in]{geometry}

\usepackage[utf8]{inputenc}		
\usepackage{jeffe}

\usepackage[charter]{mathdesign}	
\usepackage[scaled=.96,osf]{XCharter}
\usepackage[scale=0.95]{sourcesanspro}	
\usepackage{textcomp,stmaryrd}

\usepackage{graphicx}
\usepackage{cite}
\usepackage{mathtools}

\usepackage{lineno}

	\newcommand*\patchAmsMathEnvironmentForLineno[1]{%
	 \expandafter\let\csname old#1\expandafter\endcsname\csname #1\endcsname
	 \expandafter\let\csname oldend#1\expandafter\endcsname\csname end#1\endcsname
	 \renewenvironment{#1}%
	 {\linenomath\csname old#1\endcsname}%
	 {\csname oldend#1\endcsname\endlinenomath}}%
	\newcommand*\patchBothAmsMathEnvironmentsForLineno[1]{%
	 \patchAmsMathEnvironmentForLineno{#1}%
	 \patchAmsMathEnvironmentForLineno{#1*}}%
	\AtBeginDocument{%
	\patchBothAmsMathEnvironmentsForLineno{equation}%
	\patchBothAmsMathEnvironmentsForLineno{align}%
	\patchBothAmsMathEnvironmentsForLineno{flalign}%
	\patchBothAmsMathEnvironmentsForLineno{alignat}%
	\patchBothAmsMathEnvironmentsForLineno{gather}%
	\patchBothAmsMathEnvironmentsForLineno{multline}%
	}
	
\newtheorem{lemma}{Lemma}[section]
\newtheorem{theorem}[lemma]{Theorem}
\newtheorem{corollary}[lemma]{Corollary}
\numberwithin{equation}{section}

\def\Left{\operatorname{\mathit{left}}}		
\def\Right{\operatorname{\mathit{right}}}	
\def\Head{\operatorname{\mathit{head}}}		
\def\Tail{\operatorname{\mathit{tail}}}		

\def\Half{\tfrac{1}{2}}
\def\Universal#1{\widetilde{#1}}

\def\EMPH#1{\textcolor{Blue}{\textbf{\boldmath \emph{#1}}}}

\makeatletter
\renewcommand*\env@matrix[1][*\c@MaxMatrixCols c]{%
  \hskip -\arraycolsep
  \let\@ifnextchar\new@ifnextchar
  \array{#1}}
\makeatother

\def\Real{\mathbb{R}}
\def\abs#1{\mathopen| #1 \mathclose|}
\def\norm#1{\mathopen| #1 \mathclose|}
\def\seq#1{\mathopen\langle #1 \mathclose\rangle}
\def\arc#1#2{#1\mathord\shortrightarrow#2}

\usepackage{arydshln}
\def\Dx{\Delta\! x}	
\def\Dy{\Delta\! y}	
\def\Torus{\mathbb{T}}

\DeclareRobustCommand{\colvec}{\genfrac(){0pt}{}}
\newcommand{\rowvec}[2]{\left({#1},{#2}\right)}
\newcommand{\rowdet}[2]{\det\rowvec{#1}{#2}}

\newcommand{\Digamma}{\mathord{\Gamma\!\!\!\rlap{\raisebox{0.25ex}{\small-}}\,\,\,}} 

\allowdisplaybreaks	

\title{A Toroidal Maxwell--Cremona--Delaunay Correspondence%
		\thanks{Portions of this work were supported by NSF grant CCF-1408763.  A preliminary version of this paper was presented at the 36th International Symposium on Computational Geometry~\cite{el-tmcdc-20}.}}
\author{\href{http://jeffe.cs.illinois.edu}{Jeff Erickson}
		\qquad
		\href{https://patrickl.in/}{Patrick Lin}
		\\[1ex]
		University of Illinois, Urbana-Champaign}
		
\date{
Accepted to \textit{Journal of Computational Geometry} --- February 4, 2022}

\begin{document}

\begin{titlepage}

\maketitle

\begin{bigabstract}\noindent
We consider three classes of geodesic embeddings of graphs on Euclidean flat tori:
\begin{itemize}
\item
A toroidal graph embedding $\Gamma$ is \emph{positive equilibrium} if it is possible to place positive weights on the edges, such that the weighted edge vectors incident to each vertex of $\Gamma$ sum to zero.
\item
A toroidal graph embedding $\Gamma$ is \emph{reciprocal} if there is a geodesic embedding $\Gamma^*$ of its dual on the same flat torus, where each edge of $\Gamma$ is orthogonal to the corresponding dual edge in~$\Gamma^*$.
\item
A toroidal graph embedding $\Gamma$ is \emph{coherent} if it is possible to assign weights to the vertices, so that $\Gamma$ is the (intrinsic) weighted Delaunay graph of its vertices.
\end{itemize}
The classical Maxwell--Cremona correspondence and the well-known correspondence between convex hulls and weighted Delaunay triangulations imply that the analogous concepts for planar graph embeddings (with convex outer faces) are equivalent.  Indeed, all three conditions are equivalent to $\Gamma$ being the projection of the 1-skeleton of the lower convex hull of points in~$\Real^3$.  However, this three-way equivalence does not extend directly to geodesic graph embeddings on flat tori.  On any flat torus, reciprocal and coherent embeddings are equivalent, and every reciprocal embedding is in positive equilibrium, but not every positive equilibrium embedding is reciprocal.  We establish a weaker correspondence: Every positive equilibrium embedding on any flat torus is affinely equivalent to a reciprocal/coherent embedding on \emph{some} flat torus.
\end{bigabstract}
\thispagestyle{empty}

\setcounter{page}{0}
\end{titlepage}

\pagestyle{myheadings}
\markboth{A Toroidal Maxwell--Cremona--Delaunay Correspondence}
{Jeff Erickson and Patrick Lin}

\section{Introduction}

The Maxwell--Cremona correspondence is a fundamental theorem establishing an equivalence between three different structures on a straight-line embedding $\Gamma$ of a graph $G$ into the plane:

\begin{itemize}
\item An \emph{equilibrium stress} on $\Gamma$ is an assignment of non-zero weights to the edges of $G$, such that the weighted edge vectors around every interior vertex $p$ sum to zero:
\[
	\sum_{p \colon pq \in E_{\strut}} \omega_{pq} (p-q) = \colvec{0}{0}
\]
\item
A \emph{reciprocal diagram} for  $\Gamma$ is a straight-line embedding $\Gamma^*$ dual to $\Gamma$, in which every dual edge~$e^*$ is orthogonal to the corresponding primal edge $e$.
\item
A \emph{polyhedral lifting} of $\Gamma$ assigns $z$-coordinates to the vertices of $\Gamma$, so that the resulting lifted vertices in~$\Real^3$ are not all coplanar, but the lifted vertices of each face of $\Gamma$ are coplanar.
\end{itemize}

Building on earlier seminal work of Varignon \cite{v-nms-1725}, Rankine \cite{r-mam-58,r-pepf-64}, and others, Maxwell \cite{m-rffdf-70,m-atrpf-67,m-rfdf-64} proved
that any straight-line planar drawing $\Gamma$ with an equilibrium stress has both a reciprocal diagram and a polyhedral lifting.  In particular, positive and negative stresses correspond to convex and concave edges in the polyhedral lifting, respectively.  Moreover, for any equilibrium stress $\omega$ on~$\Gamma$, the vector $1/\omega$ is an equilibrium stress for the reciprocal diagram~$\Gamma^*$.  Finally, for any polyhedral liftings of~$\Gamma$, one can obtain a polyhedral lifting of the reciprocal diagram~$\Gamma^*$ via projective duality.  Maxwell's analysis was later extended and popularized by Cremona \cite{c-frnsg-72,c-gs-72}
and others; the correspondence has since been rediscovered several times in other contexts \cite{a-caecc-87,h-dcaps-77}.  More recently, Whiteley \cite{w-mspp-82} proved the converse of Maxwell's theorem: every reciprocal diagram and every polyhedral lift corresponds to an equilibrium stress; see also Crapo and Whiteley~\cite{cw-psspp-93}.  For modern expositions of the Maxwell--Cremona correspondence aimed at computational geometers, see Hopcroft and Kahn \cite{hk-prga-92}, Richter-Gebert \cite[Chapter 13]{r-rsp-96}, or Rote, Santos, and Streinu \cite{rss-ps-08}. 

If the outer face of $\Gamma$ is convex, the Maxwell--Cremona correspondence implies an equivalence between equilibrium stresses in $\Gamma$ that are \emph{positive} on every interior edge, \emph{convex} polyhedral liftings of~$\Gamma$, and reciprocal \emph{embeddings}~$\Gamma^*$.  Moreover, as Whiteley \etal~\cite{wapc-cpdts-13} and Aurenhammer~\cite{a-caecc-87} observed, the well-known equivalence between convex liftings and weighted Delaunay complexes~\cite{v-napct2-08, a-rpccc-87, a-pdpaa-87, b-vdch-79, es-vda-86} implies that all three of these structures are equivalent to a fourth:

\begin{itemize}
\item
A \emph{Delaunay weighting} of $\Gamma$ is an assignment of weights to the vertices of $\Gamma$, so that $\Gamma$ is the (power-)weighted Delaunay graph \cite{a-pdpaa-87,bs-dloss-07} of its vertices.
\end{itemize}

Among many other consequences, combining the Maxwell--Cremona correspondence with Tutte's spring-embedding theorem \cite{t-hdg-63} yields an elegant geometric proof of Steinitz's theorem \cite{sr-vtp-34,s-pr-1916} that every 3-connected planar graph is the 1-skeleton of a 3-dimensional convex polytope.  The Maxwell--Cremona correspondence has been used for scene analysis of planar drawings \cite{cw-psspp-93,s-rpfld-88,a-caecc-87,a-rpccc-87,h-dcaps-77}, finding small grid embeddings of planar graphs and polyhedra \cite{eg-dspgt-95,cgt-cdg23-96,os-qst-94,r-rsp-96,rrs-sge3-11,s-d3gvr-11,ds-esppg-17,is-dtcsg-16}, and several linkage reconfiguration problems \cite{cdr-ilstl-02,do-gfalo-07,s-prmp-06,s-eprmp-06,orsss-nfwnr-04}.

It is natural to ask how or whether these correspondences extend to graphs on surfaces other than the Euclidean plane.  Lovász \cite[Lemma~4]{l-rpcvn-01} describes a spherical analogue of Maxwell's polyhedral lifting in terms of Colin de Verdière matrices \cite{c-nigcp-90,c-ngicp-93}; see also \cite{klv-cvnsrg-97}.  Izmestiev \cite{i-skfen-19} provides a self-contained proof of the correspondence for planar frameworks, along with natural extensions to frameworks in the sphere and the hyperbolic plane.  Finally, and most closely related to the present work, Borcea and Streinu \cite{bs-lsppf-15}, building on their earlier study of rigidity in infinite periodic frameworks \cite{bs-mrpg-11,bs-pff-10}, develop an extension of the Maxwell--Cremona correspondence to infinite periodic graphs in the plane, or equivalently, to geodesic graphs on the Euclidean flat torus.  Specifically, Borcea and Streinu prove that \emph{periodic} polyhedral liftings correspond to \emph{periodic} stresses satisfying an additional homological~constraint.\footnote{Phrased in terms of toroidal frameworks, Borcea and Streinu consider only equilibrium stresses for which the corresponding reciprocal toroidal framework contains no essential cycles.  The same condition was also briefly discussed by Crapo and Whiteley \cite[Example 3.6]{cw-psspp-93}.}

\subsection{Our Results}

In this paper, we develop a different generalization of the Maxwell--Cremona--Delaunay  correspondence to geodesic embeddings of graphs on Euclidean flat tori.  Our work is inspired by and uses Borcea and Streinu's results \cite{bs-lsppf-15}, but considers a different aim.  Stated in terms of infinite periodic planar graphs, Borcea and Streinu study periodic equilibrium stresses, which necessarily include both positive and negative stress coefficients, that include periodic \emph{polyhedral lifts}; whereas, we are interested in periodic \emph{positive} equilibrium stresses that induce periodic reciprocal \emph{embeddings} and periodic \emph{Delaunay weights}.  This distinction is aptly illustrated in Figures~8–10 of Borcea and Streinu's paper \cite{bs-lsppf-15}.

Recall that a Euclidean flat torus $\Torus$ is the metric space obtained by identifying opposite sides of an arbitrary parallelogram in the Euclidean plane.  A \emph{geodesic embedding}~$\Gamma$ of a graph $G$ on the flat torus~$\Torus$ maps the vertices of $G$ to distinct points in $\Torus$ and the edges of $G$ to interior-disjoint “line segments”.  Equilibrium stresses, reciprocal embeddings, and weighted Delaunay complexes are all well-defined in the intrinsic metric of the flat torus.  We prove the following correspondences for any sufficiently well-connected geodesic embedding~$\Gamma$ on any flat torus $\Torus$.
\begin{itemize}
\item
Any equilibrium stress for $\Gamma$ is also an equilibrium stress for the affine image of $\Gamma$ on any other flat torus $\Torus'$ (Lemma~\ref{L:equ-affine}).  Equilibrium depends only on the common \emph{affine} structure of all flat tori.

\item
Any reciprocal embedding $\Gamma^*$ on $\Torus$---that is, any geodesic embedding dual to $\Gamma$ such that corresponding edges are orthogonal---defines unique equilibrium stresses in both~$\Gamma$ and~$\Gamma^*$ (Lemma~\ref{L:rec-implies-equ}).

\item
$\Gamma$ has a reciprocal embedding if and only if $\Gamma$ is 
 a weighted Delaunay complex.  Specifically, each reciprocal diagram for $\Gamma$ induces an essentially unique set of Delaunay weights for the vertices of~$\Gamma$ (Theorem~\ref{T:recip-implies-del}).  Conversely, each set of Delaunay weights for $\Gamma$ induces a \emph{unique} reciprocal diagram~$\Gamma^*$, namely the corresponding weighted Voronoi diagram (Lemma~\ref{L:coh-implies-recip}).  Thus, unlike in the plane, a reciprocal diagram $\Gamma^*$ may not be a weighted Voronoi diagram of the vertices of $\Gamma$, but some unique translation of $\Gamma^*$ is. 

\item
Unlike in the plane, $\Gamma$ may have equilibrium stresses that are not induced by reciprocal embeddings; more generally, not every equilibrium embedding on $\Torus$ is reciprocal (Theorem~\ref{Th:equ-not-recip}).  Unlike equilibrium, reciprocality depends on the \emph{conformal} structure of $\Torus$, which is determined by the shape of its fundamental parallelogram.  We derive a simple geometric condition that characterizes which equilibrium stresses are reciprocal on $\Torus$ (Lemma~\ref{L:nonsq-abg}).

\item
More generally, we show that for any equilibrium stress on $\Gamma$, there is a flat torus~$\Torus'$, unique up to rotation and scaling of its fundamental parallelogram, such that the same equilibrium stress is reciprocal for the affine image of $\Gamma$ on~$\Torus'$ (Theorem~\ref{T:skew-rec-nonlinear}).  In short, every equilibrium stress for $\Gamma$ is reciprocal on \emph{some} flat torus.  This result implies a natural toroidal analogue of Steinitz's theorem (Theorem \ref{Th:torus-steinitz}): Every essentially 3-connected torus graph~$\Gamma$ is homotopic to a weighted Delaunay complex on some flat torus.

\end{itemize}


\subsection{Other Related Results}

Our results rely on a natural generalization (Theorem~\ref{Th:tutte}) of Tutte's spring-embedding theorem to the torus, first proved (in greater generality) by Colin de Verdière \cite{c-crgtd-91}, and later proved again, in different forms, by Delgado-Friedrichs~\cite{d-eppgc-04}, Lovász~\cite[Theorem~7.1]{l-dafe-04}\cite[Theorem~7.4]{l-gg-19}, Gortler, Gotsman, and Thurston \cite{ggt-domam-06}, and (in greater generality) Hass and Scott \cite{hs-seshm-15}.
Steiner and Fischer~\cite{sf-ppc2m-04} and Gortler~\etal~\cite{ggt-domam-06} observed that this toroidal spring embedding can be computed by solving the Laplacian linear system defining the equilibrium conditions.  We describe this result and the necessary calculation in more detail in Section~\ref{S:background}.  Equilibrium and reciprocal graph embeddings can also be viewed as discrete analogues of harmonic and holomorphic functions \cite{l-gg-19,l-dafe-04}.

Our weighted Delaunay graphs are (the duals of) \emph{power diagrams} \cite{a-pdpaa-87,ai-gravd-88} or \emph{Laguerre-Voronoi diagrams}~\cite{iim-vdlga-85} in the intrinsic metric of the flat torus.  Toroidal Delaunay triangulations are commonly used to generate finite-element meshes for simulations with periodic boundary conditions, and several efficient algorithms for constructing these triangulations are known \cite{mr-vdo-97,gm-cgs-01,ct-dtced-16,btv-dtosl-16}.  Building on earlier work of Rivin \cite{r-esssh-94} and Indermitte \etal~\cite{iltc-vdpfs-01}, Bobenko and Springborn \cite{bs-dloss-07} proved that on any piecewise-linear surface, intrinsic Delaunay triangulations can be constructed by an intrinsic incremental flipping algorithm, mirroring the classical planar algorithm of Lawson \cite{l-tt-72}; their analysis extends easily to intrinsic weighted Delaunay graphs.  Weighted Delaunay complexes are also known as \emph{regular} or \emph{coherent} subdivisions \cite{z-lp-95,drs-tsaa-10}.

Finally, equilibrium and reciprocal embeddings are closely related to the celebrated Koebe-Andreev circle-packing theorem: Every planar graph is the contact graph of a set of interior-disjoint circular disks \cite{k-kka-36,a-cpls-70,a-cpfvl-70}; see Felsner and Rote \cite{fr-pdcpr-18} for a simple proof, based in part on earlier work of Brightwell and Scheinerman \cite{bs-rpg-93} and Mohar \cite{m-ptcpa-93}.  The circle-packing theorem was generalized to higher-genus surfaces by Colin de Verdière \cite{c-eccdm-89,c-pvpec-91} and Mohar \cite{m-cpmec-97,m-cpmpt-97}.  In particular, Mohar proved that any well-connected graph embedding on the torus is \emph{homotopic} to the contact graph of an essentially unique circle packing \emph{for a unique Euclidean metric} on the torus.  This disk-packing representation immediately yields a weighted Delaunay graph, where the areas of the disks are the vertex weights.  We revisit and extend this result in Section~\ref{S:Steinitz}.

Discrete harmonic and holomorphic functions, circle packings, and intrinsic Delaunay triangulations have numerous applications in discrete differential geometry; we refer the reader to monographs by Crane \cite{c-ddgai-19}, Lovász \cite{l-gg-19}, and Stephenson \cite{s-icptd-05}.

\section{Background and Definitions}
\label{S:background}

\subsection{Flat Tori}

A \EMPH{flat torus} is the metric surface obtained by identifying opposite sides of a  parallelogram in the Euclidean plane.  Specifically, for any nonsingular $2\times 2$ matrix $M = \begin{psmallmatrix} a & b \\ c & d \end{psmallmatrix}$, let~\EMPH{$\Torus_M$} denote the flat torus obtained by identifying opposite edges of the \EMPH{fundamental parallelogram~$\lozenge_M$} with vertex coordinates $\colvec{0}{0}$, $\colvec{a}{c}$, $\colvec{b}{d}$, and $\colvec{a+b}{c+d}$.  In particular, the \emph{square} flat torus \EMPH{$\Torus_\square$} $= \Torus_I$ is obtained by identifying opposite sides of the Euclidean unit square~$\square = \lozenge_I = [0,1]^2$.  The linear map $M\colon \Real^2\to\Real^2$ naturally induces a homeomorphism~$\underline{M}$ from $\Torus_\square$ to $\Torus_M$.

Equivalently, $\Torus_M$ is the quotient space of the plane $\Real^2$ with respect to the lattice of translations generated by the columns of $M$; in particular, the square flat torus is the quotient space $\Real^2/\Z^2$.  The quotient map $\pi_M \colon \Real^2 \to \Torus_M$ is called a \emph{covering} map or \EMPH{projection}.  A \EMPH{lift} of a point $p\in \Torus_M$ is any point in the preimage $\pi_M^{-1}(p) \subset \Real^2$.  A \EMPH{geodesic} in $\Torus_M$ is the projection of any line segment in $\Real^2$; we emphasize that geodesics are \emph{not} necessarily shortest paths.  A \EMPH{closed geodesic} is a geodesic whose endpoints coincide.

\subsection{Graphs, Drawings, and Embeddings}

We regard each edge of an undirected graph $G$ as a pair of opposing \EMPH{darts}, each directed from one endpoint, called the \EMPH{tail} of the dart, to the other endpoint, called its \EMPH{head}.  For each edge $e$, we arbitrarily label the darts \EMPH{$e^+$} and \EMPH{$e^-$}; we call $e^+$ the \EMPH{reference dart} of~$e$.  We explicitly allow graphs with loops and parallel edges.  At the risk of confusing the reader, we often write $\arc{p}{q}$ to denote an arbitrary dart with tail $p$ and head $q$, and $\arc{q}{p}$ for the reversal of $\arc{p}{q}$.

A \EMPH{drawing} of a graph $G$ on a torus $\Torus$ is any continuous function $\Gamma$ from $G$ (as a topological space) to $\Torus$, which maps vertices of $G$ to points in $\Torus$ and edges of $G$ to curves between the images of their endpoints.  A \emph{vertex} of a drawing $\Gamma$ is the image $\Gamma(p)$ of some vertex $p$ of $G$; similarly, an \emph{edge} of $\Gamma$ is the curve $\Gamma(e)$ (formally, the restriction~$\Gamma|_e\colon {[0,1]\to\Torus}$) for some edge $e$ of~$G$.

An \EMPH{embedding} is an injective drawing, which maps vertices to \emph{distinct} points and edges to \emph{simple, interior-disjoint} curves.  The \EMPH{faces} of an embedding are the components of the complement of the image of the graph; we consider only \emph{cellular} embeddings, in which all faces are open disks.  In any embedding, \EMPH{$\Left(d)$} and \EMPH{$\Right(d)$} denote the faces immediately to the left and right of (the image of) any dart~$d$; these could be the same face.  The complex of vertices, edges, and faces induced by a cellular  embedding is called a \EMPH{map}.  

Let $\Gamma$ be any drawing of a graph $G$ on the torus $\Torus_M$.  The \EMPH{universal cover} of the drawing $\Gamma$ is the unique infinite planar graph drawing $\Universal\Gamma \colon\Universal{G}\into \Real^2$ that is periodic with respect to the lattice generated by the columns of $M$, such that the covering map $\pi_M$ projects the image of $\Universal\Gamma$ onto the image of~$\Gamma$.  (The infinite graph~$\Universal{G}$ is a function of the embedding $\Gamma$, not just the underlying graph $G$.)  In particular, the covering map~$\pi_M$ projects each vertex or edge of $\Universal\Gamma$ to a vertex or edge of~$\Gamma$, respectively, and each vertex of $\Gamma$ lifts to an infinite lattice of vertices of $\Universal\Gamma$.  Moreover, if $\Gamma$ is an embedding, then its universal cover $\Universal\Gamma$ is also an embedding, and the covering map $\pi_M$ projects each face of $\Universal\Gamma$ to a face of~$\Gamma$.  

A drawing $\Gamma$ is \EMPH{geodesic} if it maps each edge to a geodesic, or equivalently, if its universal cover $\Universal\Gamma$ maps each edge to a straight line segment.


We call $\Gamma$ \EMPH{essentially simple} if the graph $\Universal{G}$ is simple, and \EMPH{essentially 3-connected} if $\Universal{G}$ is 3-connected \cite{m-cpmec-97,m-cpmpt-97,mr-tvrmt-98,ms-bns-03,gl-tmswo-14}.  We emphasize that essential simplicity and essential 3-connectivity are features of \emph{embeddings}, not abstract graphs; Figure~\ref{F:torus-graph} shows an essentially simple, essentially 3-connected toroidal embedding of a graph that is neither simple nor 3-connected.  

\begin{figure}[htb]
\centering
\raisebox{-0.5\height}{\includegraphics[scale=0.66]{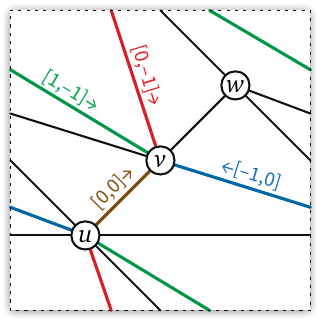}}\qquad
\raisebox{-0.5\height}{\includegraphics[scale=0.33]{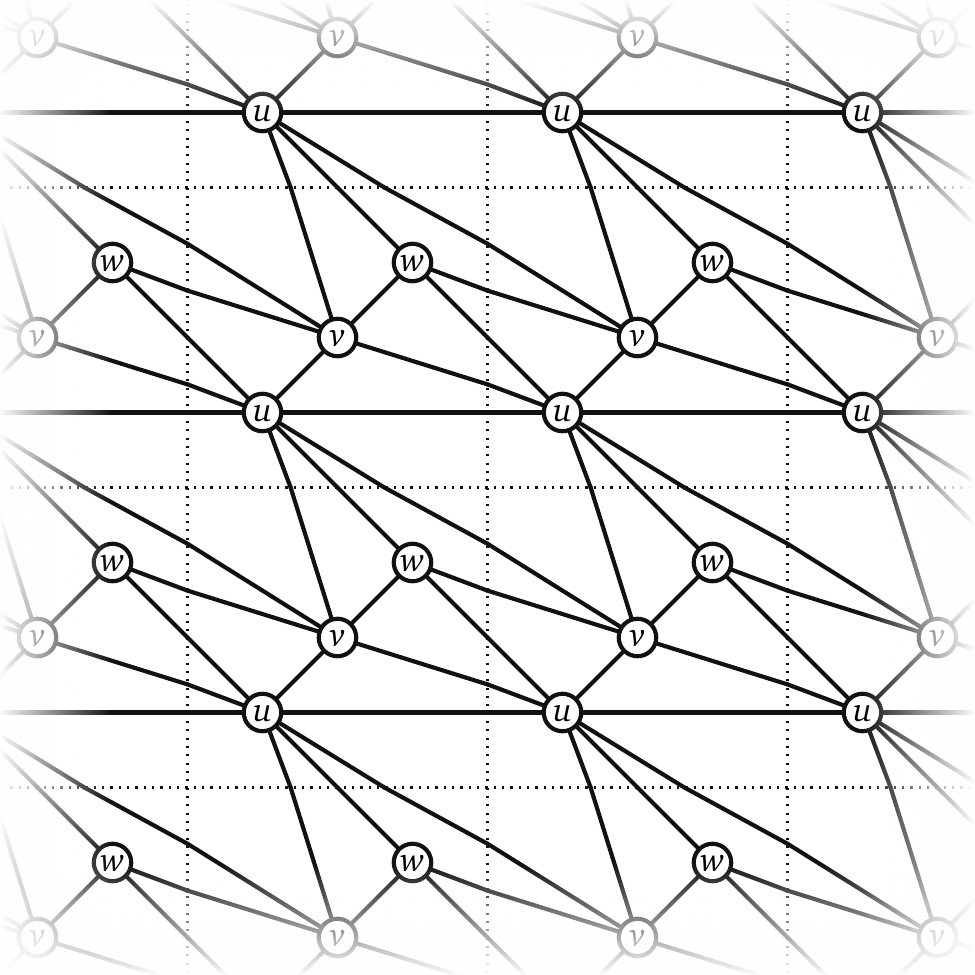}}\qquad
\raisebox{-0.5\height}{\includegraphics[scale=0.66]{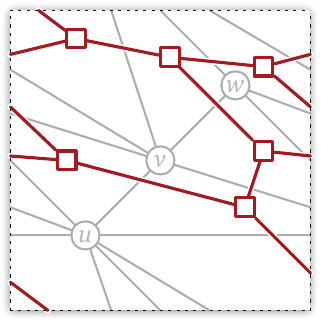}}
\caption{An essentially simple, essentially 3-connected geodesic graph embedding on the square flat torus (showing the homology vectors of all four darts from $u$ to $v$), a small portion of its universal cover, and a geometric dual embedding.}
\label{F:torus-graph}
\end{figure}

\subsection{Homology and Circulations}

For any drawing $\Gamma$ of a graph $G$ on the square flat torus $\Torus_\square$, we associate a \EMPH{homology vector}~$[d]_\Gamma \in \Z^2$ with each dart $d$, which records how the dart crosses the boundary edges of the unit square.  Specifically, the first coordinate of $[d]_\Gamma$ is the number of times the curve $\Gamma(d)$ crosses the vertical boundary rightward, minus the number of times $\Gamma(d)$ crosses the vertical boundary leftward; and the second coordinate of~$[d]_\Gamma$ is the number of times $\Gamma(d)$ crosses the horizontal boundary upward, minus the number of times $\Gamma(d)$ crosses the horizontal boundary downward.  In particular, reversing a dart negates its homology vector: $[e^+]_\Gamma = -[e^-]_\Gamma$.  Again, see Figure~\ref{F:torus-graph}.  For drawings on any other flat torus~$\Torus_M$, homology vectors are similarly defined by how darts cross the boundary of the fundamental parallelogram $\lozenge_M$.  

The (integer) \EMPH{homology class} $[\gamma]_\Gamma$ of a directed cycle $\gamma$ in $G$, with respect to the drawing $\Gamma$, is the sum of the homology vectors of its forward darts.  A cycle is \emph{contractible} in $\Gamma$ if its homology class is $\colvec{0}{0}$ and \emph{essential} otherwise.  In particular, the boundary cycle of each face of $\Gamma$ is contractible.

A \EMPH{circulation} $\phi$ in $G$ is a function from the darts of $G$ to the reals, such that~$\phi(\arc{p}{q}) = -\phi(\arc{q}{p})$ for every dart $\arc{p}{q}$ and 
\(
	\sum_{\arc{p}{q}} \phi(\arc{p}{q}) = 0
\)
for every vertex $p$.  We represent circulations by column vectors in $\Real^E$, indexed by the edges of $G$, where $\phi_e = \phi(e^+)$.  Let \EMPH{$\Lambda_\Gamma$} denote the $2\times E$ matrix whose columns are the homology vectors of the reference darts in~$G$.  The (real) homology class of a circulation is the matrix-vector product
\[
	[\phi]_\Gamma ~=~ \Lambda_\Gamma \phi ~=~ \sum_{e\in E} \phi(e^+)\cdot [e^+]_\Gamma.
\]
(This identity directly generalizes our earlier definition of the homology class $[\gamma]_\Gamma$ of a cycle~$\gamma$.)

We will omit the subscript $\Gamma$ from our homology notation when the drawing is clear from context.

\subsection{Homotopy}

Two closed curves $\gamma$ and $\gamma'$ on a torus $\Torus$ are \EMPH{homotopic} if one can be continuously deformed into the other, or more formally, if there is a continuous family $(\gamma_t)_{t\in[0,1]}$ of closed curves~$\gamma_t\colon S^1\to\Torus$ such that $\gamma_0=\gamma$ and $\gamma_1=\gamma'$.  Similarly, two drawings of the same graph~$G$ on the same flat torus $\Torus$ are homotopic if one can be continuously deformed into the other through drawings of the same graph.  The continuous family of cycles or drawings is called a \EMPH{homotopy}.

Two closed curves on any torus are homotopic if and only if they have the same homology class.  The following lemma characterizes when two drawings are homotopic; a similar characterization of \emph{isotopic embeddings} was proved by Ladegaillerie \cite{l-ctp1c-74a,l-ctp1c-74b,l-cdp1c-84,cm-tgis-14}.\footnote{Two embeddings are isotopic if and only if they are connected by a continuous family of embeddings.  Results of Ladegallierie \cite[Theorem~3.1]{l-cdp1c-84} and Mohar~\cite[p.~42]{m-eig-88} imply that two essentially simple, essentially 3-connected embeddings of the same graph are isotopic if and only if they are homotopic.}

\begin{lemma}
\label{L:homotopy}
Two drawings $\Gamma$ and $\Gamma'$ of the same graph $G$ on the same flat torus~$\Torus$ are homotopic if and only if every cycle in $G$ has the same homology class in both drawings.
\end{lemma}

\begin{proof}
Suppose $\Gamma$ and $\Gamma'$ are homotopic.  For any cycle $\gamma$ in $G$, the homotopy from $\Gamma$ to~$\Gamma'$ restricts to a homotopy from $\Gamma(\gamma)$ to $\Gamma'(\gamma)$.  It follows that $\gamma$ has the same homology class in both drawings.

Conversely, suppose $[\gamma]_\Gamma = [\gamma]_{\Gamma'}$ for every cycle $\gamma$ in $G$.  We construct a homotopy from $\Gamma$ to $\Gamma'$ as follows.  Let $T$ be any spanning tree of $G$.  We first continuously deform $\Gamma$ to an intermediate geodesic drawing $\Gamma_*$ by contracting $T$ to a point, translating that point to some fixed location $p$, and then deforming the non-tree edges of $G$ to geodesic loops.  We can similarly homotope $\Gamma'$ to an intermediate geodesic drawing~$\Gamma'_*$.  To complete the proof, it suffices to show that $\Gamma_*$ and $\Gamma'_*$ are identical drawings.

Let $e$ be an arbitrary edge of $G$.  If $e$ is in the spanning tree $T$, then the image of $e$ in both drawings is the point $p$.  Otherwise, let $\gamma_e$ be the unique fundamental cycle in $T+e$.  The homotopy from~$\Gamma$ to $\Gamma_*$ preserves the homology class of $\gamma_e$.  Thus, $\Gamma_*(e) = \Gamma_*(\gamma_e)$ is a geodesic cycle through $p$ with homology class $[\gamma_e]_\Gamma$.  Similarly, $\Gamma'_*(e)$ is a geodesic cycle through $p$ in the same homology class.  Each homology class of cycles on $\Torus$ contains a unique closed geodesic through any point.  In all cases, we conclude that $\Gamma_*(e) = \Gamma'_*(e)$.
\end{proof}

\subsection{Geodesic Drawings and Embeddings}
\label{sec:embedrep}
%


Recall that a graph drawing is \emph{geodesic} if it maps edges to geodesics.  Any geodesic drawing~$\Gamma\colon G \to \Torus_M$ is uniquely determined by its \EMPH{coordinate representation}, which consists of a coordinate vector $\seq{p}_\Gamma\in \lozenge_M$ for each vertex $p$, together with the homology vector~$[e^+]_\Gamma \in \Z^2$ of each edge $e$.  (Again, we will omit the subscript $\Gamma$ when the drawing is clear from context.) 

The \EMPH{displacement vector $\Delta_d$} of any dart $d$ (with respect to a fixed drawing $\Gamma$) is the difference between the head and tail coordinates of any lift of $d$ in the universal cover $\Universal{\Gamma}$.  Displacement vectors can be equivalently defined in terms of vertex coordinates, homology vectors, and the shape matrix $M$ as follows:
\[
	\Delta_{\arc{p}{q}} := \seq{q} - \seq{p} + M\, [{\arc{p}{q}}].
\]
Reversing a dart negates its displacement: $\Delta_{\arc{q}{p}} = -\Delta_{\arc{p}{q}}$.
We sometimes write~$\Dx_d$ and~$\Dy_d$ to denote the first and second coordinates of $\Delta_d$.  The \EMPH{displacement matrix}~$\Delta = \Delta(\Gamma)$ of a geodesic drawing $\Gamma$ is the $2\times E$ matrix whose columns are the displacement vectors of the reference darts.  Every geodesic drawing on $\Torus_M$ is determined \emph{up to translation} by its displacement matrix.  Finally, we let \EMPH{$\abs{e}$} denote the length of any edge $e$, or equivalently, the Euclidean length of the displacement vectors $\Delta_{e^\pm}$.

On the \emph{square} flat torus, the integer homology class of any directed cycle is also equal to the sum of the \emph{displacement} vectors of its darts:
\[
	[\gamma] ~=~ \sum_{\arc{p}{q}\in\gamma} [\arc{p}{q}] ~=~ \sum_{\arc{p}{q}\in\gamma} \Delta_{\arc{p}{q}}.
\]
In particular, the total displacement of any contractible cycle is zero, as expected. Extending this identity to circulations by linearity gives us the following useful lemma:

\begin{lemma}
\label{L:harmonic}
Let $\Gamma\colon G\to \Torus_\square$ be any geodesic drawing of a graph $G$ on the square flat torus, and let $\Delta = \Delta(\Gamma)$ be the displacement matrix of~$\Gamma$.  For every circulation~$\phi$ in $G$, we have~$\Delta \phi = \Lambda_\Gamma \phi = [\phi]_\Gamma$.
\end{lemma}
%
%

The following lemma is essentially the converse of Lemma~\ref{L:harmonic}.

\begin{lemma}
\label{L:draw-on-square}
Fix an essentially simple, essentially 3-connected embedding $\Gamma\colon G\to \Torus_\square$ and a $2 \times E$ matrix~$\Delta$.  Suppose for every directed cycle (and therefore every circulation) $\phi$ in~$G$, we have $\Delta\phi = \Lambda_\Gamma\phi = [\phi]_\Gamma$.  Then $\Delta$ is the displacement matrix of a geodesic drawing $\Gamma'\colon {G\to \Torus_\square}$ that is homotopic to $\Gamma$.
\end{lemma}

\begin{proof}
We can construct a non-standard coordinate representation of a drawing~$\Gamma'$ whose displacement matrix is $\Delta$ as follows.  Fix an arbitrary spanning tree $T$ of~$G$, rooted at an arbitrary vertex $r$.  Define $\seq{r} = \colvec{0}{0}$, and for any other vertex $q$ with parent $p$ in $T$, define $\seq{q} = \seq{p} + \Delta_{\arc{p}{q}}$ and $[\arc{p}{q}] = [\arc{q}{p}] = \colvec{0}{0}$.  Finally, for every edge $e$ that is not in $T$, let~$\gamma^+_e$ denote the unique fundamental cycle in the subgraph $T+e$, directed so that it contains the reference dart $e^+$, and define $[e^+] = [\gamma^+_e]_\Gamma$ and $[e^-] = -[\gamma^+_e]_\Gamma$.

The vectors $\seq{p}$ and $[\arc{p}{q}]$ we just defined do not necessarily fit our definition of coordinate representation; in particular, the coordinate vectors $\seq{p}$ can fall outside the unit square.  However, we can modify these vectors to fit our earlier definitions, intuitively by folding the coordinate vectors into the unit square, as follows.\footnote{See the appendix of Chambers \etal~\cite{celp-hmgt-21} and references therein for a discussion of equivalent non-standard coordinate representations.}  For any vector $u = \colvec{a}{b}$, let $\floor{u}$ denote the vector $\colvec{\floor{a}}{\floor{b}}$.  For each vertex~$p$ of $G$, define
\[
	\seq{p}_{\Gamma’} := \seq{p} - \floor{\seq{p}};
\]
we immediately have $\seq{p}_{\Gamma’} \in [0,1)^2$.  Similarly, for each dart $\arc{p}{q}$ of $G$, define
\[
	[\arc{p}{q}]_{\Gamma’} := [\arc{p}{q}] - \floor{\seq{p}} + \floor{\seq{q}}.
\]
Finally, let $\Gamma'\colon {G\to \Torus_\square}$ be the unique geodesic drawing of $G$ with coordinate vectors $\seq{p}_{\Gamma’}$ and homology vectors $[\arc{p}{q}]_{\Gamma’}$.

By construction, for every dart $\arc{p}{q}$, we have
\[
	\Delta_{\arc{p}{q}} 
		~=~ \seq{q} - \seq{p} + [\arc{p}{q}]
		~=~ \seq{q}_{\Gamma'} - \seq{p}_{\Gamma'} + [\arc{p}{q}]_{\Gamma’},
\]
so $\Delta$ is in fact the displacement matrix of $\Gamma'$.  For every fundamental cycle $\gamma = \gamma^+_e$, we immediately have
\[
	[\gamma]_{\Gamma'}
	~:=~
	\sum_{\arc{p}{q}\in \gamma} [\arc{p}{q}]_{\Gamma'}
	~=~
	\sum_{\arc{p}{q}\in \gamma} \Delta_{\arc{p}{q}}
	~=~
	[\gamma]_{\Gamma}.
\]
It follows by linearity that $[\phi]_{\Gamma'} = [\phi]_\Gamma$ for every circulation~$\phi$ in $G$.  Thus, Lemma \ref{L:homotopy} implies that~$\Gamma'$ is homotopic to $\Gamma$.
\end{proof}

\subsection{Equilibrium Stresses and Spring Embeddings}

A \EMPH{stress} in a geodesic torus drawing $\Gamma\colon G\to \Torus$ is a real vector $\omega \in \R^E$  indexed by the edges of $G$.  Unlike homology vectors, circulations, and displacement vectors, stresses can be viewed as \emph{symmetric} functions on the darts of~$G$.  An \EMPH{equilibrium stress} in $\Gamma$ is a stress $\omega$ that satisfies the following identity at every vertex $p$:
\[
	\sum_{\arc{p}{q}} \omega_{pq} \Delta_{\arc{p}{q}}
	=
	\colvec{0}{0}.
\]
Unlike Borcea and Streinu \cite{bs-lsppf-15,bs-mrpg-11,bs-pff-10}, we primarily consider \EMPH{positive} equilibrium stresses, where~$\omega_e>0$ for every edge $e$.  It may be helpful to imagine each stress coefficient $\omega_e$ as a linear spring constant; intuitively, each edge pulls its endpoints inward, with a force equal to the length of $e$ times the stress coefficient~$\omega_e$.

Recall that the linear map $M\colon \R^2\times \R^2$ associated with any nonsingular $2\times 2$ matrix induces a homeomorphism $\underline{M}\colon \Torus_\square \to \Torus_M$.  In particular, this homeomorphism transforms a geodesic embedding $\Gamma\colon G\to \Torus_\square$ with displacement matrix $\Delta$ into a geodesic embedding ${\underline{M}\circ \Gamma} \colon G\to\Torus_M$ with displacement matrix $M\Delta$.  We refer to the embedding $\underline{M}\circ \Gamma$ as the \EMPH{affine image} of $\Gamma$ on~$\Torus_M$.  Routine definition-chasing now implies the following lemma.

\begin{lemma}
\label{L:equ-affine}
Let $\Gamma\colon G\to\Torus_\square$ be a geodesic drawing on the square flat torus $\Torus_\square$.   If $\omega$ is an equilibrium stress for $\Gamma$, then $\omega$ is also an equilibrium stress for its affine image on any other flat torus~$\Torus_M$.
\end{lemma}

Our results rely on the following natural generalization of Tutte's spring embedding theorem to flat torus embeddings.

\begin{theorem}[Colin de Verdière~\cite{c-crgtd-91}; see also \cite{d-eppgc-04, l-dafe-04, ggt-domam-06, hs-seshm-15}]
\label{Th:tutte}
Let $\Gamma$ be any essentially simple, essentially 3-connected embedding on any flat torus $\Torus$, and let $\omega$ be any positive stress on the edges of~$\Gamma$.  Up to translation, there is a unique drawing homotopic to $\Gamma$ that is in equilibrium with respect to $\omega$, and that drawing is an embedding with convex faces.
\end{theorem}

Lemma \ref{L:draw-on-square} and Theorem \ref{Th:tutte} immediately imply the following sufficient condition for a displacement matrix to describe a geodesic \emph{embedding} on the square torus.  

\begin{corollary}
\label{C:embed-on-square}
Fix an essentially simple, essentially 3-connected embedding $\Gamma\colon G\to \Torus_\square$, a $2 \times E$ matrix~$\Delta$.  Suppose for every directed cycle $\phi$ in~$G$, we have $\Delta\phi = \Lambda_\Gamma\phi = [\phi]_\Gamma$.  Then~$\Delta$ is the displacement matrix of a geodesic drawing $\Gamma'\colon G\to \Torus_\square$ that is homotopic to~$\Gamma$.  If in addition $\Gamma'$ has a positive equilibrium stress, then $\Gamma'$ is an embedding.
\end{corollary}
%
%

Following Steiner and Fischer~\cite{sf-ppc2m-04} and Gortler, Gotsman, and Thurston \cite{ggt-domam-06}, given the coordinate representation of any geodesic embedding $\Gamma$ on the square flat torus, with any positive stress vector $\omega>0$, we can compute an equilibrium embedding homotopic to~$\Gamma$ by solving the linear system
\begin{equation}
	\sum_{\arc{p}{q}} \omega_{pq} \big(\seq{q} - \seq{p} + [{\arc{p}{q}}]\big)
	=
	\colvec{0}{0} \qquad\text{for every vertex $q$}
	\label{Eq:tutte}
\end{equation}
for the vertex locations $\seq{p}$, treating the homology vectors $[{\arc{p}{q}}]$ as constants.  Alternatively, Corollary~\ref{C:embed-on-square} implies that we can compute the displacement vectors of an equilibrium embedding homotopic to $\Gamma$ directly, by solving the linear system
\begin{equation*}
	\begin{aligned}
		\sum_{\arc{p}{q}} \omega_{pq} \Delta_{\arc{p}{q}}
		&= \colvec{0}{0}
		\qquad\text{for every vertex $q$}
	\\
		\sum_{\Left(d) = f} \Delta_d
		&= \colvec{0}{0}
		\qquad\text{for every face $f$}
	\\
		\sum_{d \in \gamma_1} \Delta_d &= [\gamma_1]
	\\
		\sum_{d \in \gamma_2} \Delta_d &= [\gamma_2]
	\end{aligned}
\end{equation*}
where $\gamma_1$ and $\gamma_2$ are any two directed cycles with independent non-zero homology classes.

\subsection{Duality and Reciprocality}
\label{SS:duality}

Two embeddings $\Gamma\colon G\to\Torus$ and $\Gamma^*\colon G^*\to\Torus$ are \EMPH{geometric duals} if there is a bijection between the edges of $G$ and $G^*$, such that each edge $e$ of $\Gamma$ crosses only the corresponding edge $e^*$ of $\Gamma^*$, each vertex $v$ of $\Gamma$ lies inside a unique face $v^*$ of $\Gamma^*$, and each face $f$ of $\Gamma$ contains a unique vertex $f^*$ of $\Gamma^*$.  Each dart $d$ in $G$ has a corresponding dart $d^*$ in~$G^*$, defined by setting $\Head(d^*) = \Left(d)^*$ and $\Tail(d^*) = \Right(d^*)$; intuitively, the dual of a dart in $\Gamma$ is obtained by rotating the dart counterclockwise.  More generally, we call two embeddings $\Gamma$ and $\Gamma^*$ \EMPH{duals} if $\Gamma^*$ is homotopic to a geometric dual of $\Gamma$.  We emphasize that if $\Gamma$ and $\Gamma^*$ are dual embeddings, an edge $e$ in $\Gamma$ need not intersect the corresponding dual edge $e^*$ in $\Gamma^*$.

Two dual geodesic embeddings $\Gamma$ and $\Gamma^*$ on the same flat torus $\Torus$ are \EMPH{reciprocal} if every edge $e$ of~$\Gamma$ is orthogonal to its dual edge $e^*$ in $\Gamma^*$.  Again, we emphasize that $e$ and $e^*$ may not intersect.  We call a single embedding $\Gamma$ \EMPH{reciprocal} if it is reciprocal to some dual embedding $\Gamma^*$. 


It will prove convenient to treat vertex coordinates, displacement vectors, homology vectors, and circulations in any dual embedding~$\Gamma^*$ as \emph{row} vectors.  For any vector $v\in\Real^2$ we define \EMPH{$v^\perp$}${} := (Jv)^T$, where $\EMPH{J} := \begin{psmallmatrix} 0&-1\\1&~0 \end{psmallmatrix}$ is the matrix for a $90^\circ$ counterclockwise rotation.  Note that $J^T = J^{-1} = -J$.  Similarly, for any $2\times n$ matrix~$A$, we define \EMPH{$A^\perp$}${} := (JA)^T = -A^T J$.

\subsection{Coherent Subdivisions}

Let $\Gamma\colon G\to\Torus_M$ be a fixed geodesic embedding, and fix arbitrary real weights \EMPH{$\pi_p$} for every vertex $p$ of~$G$.  Let $\arc{p}{q}$, $\arc{p}{r}$, and $\arc{p}{s}$ be three consecutive darts around a common tail $p$ in counterclockwise order, so that $\Left(\arc{p}{q}) = \Right(\arc{p}{r})$ and $\Left(\arc{p}{r}) = \Right(\arc{p}{s})$.  We call the edge $pr$ \EMPH{locally Delaunay} if the following determinant is positive:
\begin{equation}
	\begin{vmatrix}
		\Dx_{\arc{p}{q}} & \Dy_{\arc{p}{q}} &
			\Half\norm{\Delta_{\arc{p}{q}}}^2 + \pi_p - \pi_q\\[0.5ex]
		\Dx_{\arc{p}{r}} & \Dy_{\arc{p}{r}} &
			\Half\norm{\Delta_{\arc{p}{r}}}^2 + \pi_p - \pi_r\\[0.5ex]
		\Dx_{\arc{p}{s}} & \Dy_{\arc{p}{s}} &
			\Half\norm{\Delta_{\arc{p}{s}}}^2 + \pi_p - \pi_s\\
	\end{vmatrix} > 0.
	\label{Eq:Delaunay}
\end{equation}
This inequality follows by elementary row operations and cofactor expansion from the standard determinant test for appropriate lifts of the vertices $p,q,r,s$ to the universal cover:
\begin{equation}
	\begin{vmatrix}
		1 & x_p & y_p & \Half(x_p^2 + y_p^2) - \pi_p  \\[0.5ex]
		1 & x_q & y_q & \Half(x_q^2 + y_q^2) - \pi_q  \\[0.5ex]
		1 & x_r & y_r & \Half(x_r^2 + y_r^2) - \pi_r  \\[0.5ex]
		1 & x_s & y_s & \Half(x_s^2 + y_s^2) - \pi_s  \\
	\end{vmatrix}
	>
	0.
	\label{Eq:Delaunay2}
\end{equation}
(The factor~$1/2$ simplifies our later calculations, and is consistent with Maxwell's construction of polyhedral liftings and reciprocal diagrams.)  Similarly, we say that an edge is \EMPH{locally flat} if the corresponding determinant is zero.  Finally, $\Gamma$ is the \EMPH{weighted Delaunay complex} of its vertices if every edge of $\Gamma$ is locally Delaunay and every diagonal of every non-triangular face is locally flat.

Equivalently, $\Gamma$ is the weighted Delaunay complex of its vertices if and only if $\Gamma$ is the projection of the weighted Delaunay complex of the lift $\pi_M^{-1}(V)$ of its vertices $V$ to the universal cover.  Results of Bobenko and Springborn \cite{bs-dloss-07} imply that any finite set of weighted points on any flat torus has a unique weighted Delaunay complex.  We emphasize that weighted Delaunay complexes are \emph{not} necessarily either simple or triangulations; however, every weighted Delaunay complex on any flat torus is both essentially simple and essentially 3-connected.  The dual \EMPH{weighted Voronoi diagram} of a set of points $P\subset \Torus_M$, also known as its \emph{power diagram} \cite{a-pdpaa-87,ai-gravd-88}, can be defined similarly by projection from the universal cover. 

Finally, a geodesic torus embedding is \EMPH{coherent} if it is the weighted Delaunay complex of its vertices, with respect to some vector of weights.

\section{Reciprocal Implies Equilibrium}

Fix two dual geodesic embeddings $\Gamma \colon G\to \Torus_M$ and $\Gamma^* \colon G^* \to \Torus_M$.  We write $\abs{e}$ to denote the length of any edge $e$ in $\Gamma$ and $\abs{e^*}$ to denote the length of the corresponding dual edge~$e^*$ in $\Gamma^*$.

\begin{lemma}
\label{L:rec-implies-equ}
Suppose $\Gamma$ and $\Gamma^*$ are reciprocal geodesic embeddings on $\Torus_M$.  Then the vector~$\omega$ defined by $\omega_e = \abs{e^*}/\abs{e}$ is an equilibrium stress for $\Gamma$, and  symmetrically, the vector $\omega^*$ defined by $\omega^*_{e^*} = 1/\omega_e = \abs{e}/\abs{e^*}$ is an equilibrium stress for $\Gamma^*$.
\end{lemma}

\begin{proof}
Let $\omega_e = \abs{e^*}/\abs{e}$ and $\omega^*_{e^*} = 1/\omega_e = \abs{e}/\abs{e^*}$ for each edge $e$.  Let $\Delta$ denote the displacement matrix of $\Gamma$, and let~$\Delta^*$ denote the (transposed) displacement matrix of $\Gamma^*$.  We immediately have $\Delta^*_{e^*} = \omega_e \Delta_{e}^\perp$ for every edge $e$ of $\Gamma$.  The darts leaving each vertex $p$ of~$\Gamma$ dualize to a facial cycle around the corresponding face $p^*$ of $\Gamma^*$, and thus
\[
	\left(\sum_{q \colon pq\in E} \omega_{pq} \Delta_{\arc{p}{q}} \right)^\perp
	= \sum_{q \colon pq\in E} \omega_{pq} \Delta_{\arc{p}{q}}^\perp
	= \sum_{q \colon pq\in E} \Delta^*_{(\arc{p}{q})^*}
	= \rowvec{0}{0}.
\]
We conclude that $\omega$ is an equilibrium stress for $\Gamma$, and thus (by swapping the roles of $\Gamma$ and~$\Gamma^*$) that~$\omega^*$ is an equilibrium stress for $\Gamma^*$.
\end{proof}

A \emph{positive} stress vector $\omega$ is a \EMPH{reciprocal stress} for $\Gamma$ if there is a reciprocal embedding $\Gamma^*$ on the same flat torus such that $\omega_e = \abs{e^*}/\abs{e}$ for each edge $e$.  Thus, a geodesic toroidal embedding is reciprocal if and only if it has a reciprocal stress, and Lemma~\ref{L:rec-implies-equ} implies that every reciprocal stress is an equilibrium stress.  The following simple construction shows that the converse of Lemma~\ref{L:rec-implies-equ} is false.

\begin{theorem}
\label{Th:equ-not-recip}
Not every positive equilibrium stress for a geodesic toroidal embedding $\Gamma$ is a reciprocal stress.  More generally, not every equilibrium embedding on $\Torus$ is reciprocal/coherent on~$\Torus$.
\end{theorem}

\begin{proof}
Let $\Gamma_1$ be the geodesic triangulation in the flat square torus $\Torus_\square$ with a single vertex~$p$ and three edges, whose reference darts have displacement vectors $\colvec{1}{0}$, $\colvec{1}{1}$, and $\colvec{2}{1}$.  Every stress $\omega$ in~$\Gamma$ is an equilibrium stress, because the forces applied by each edge cancel out.  The weighted Delaunay complex of a single point is identical for all weights, so it suffices to verify that $\Gamma_1$ is not an intrinsic Delaunay triangulation.  We easily observe that the longest edge of $\Gamma_1$ is not Delaunay.  See Figure~\ref{F:incoherent}. 

\begin{figure}[ht]
\centering\includegraphics[scale=0.6]{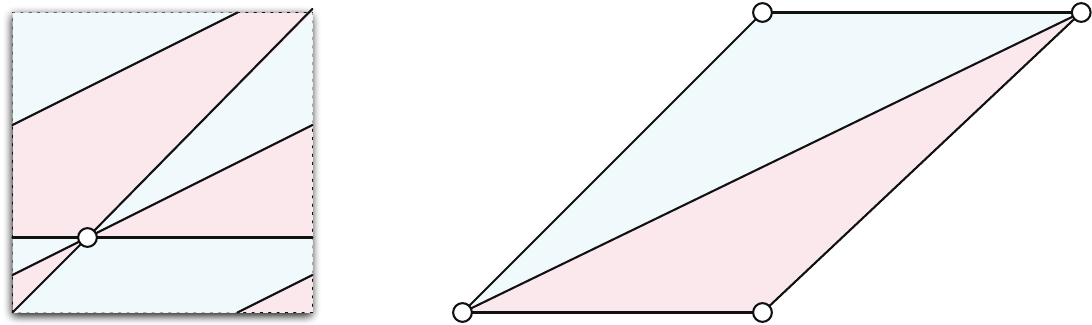}
\caption{A one-vertex triangulation $\Gamma_1$ on the square flat torus, and a lift of its faces to the universal cover.  Every stress in $\Gamma_1$ is an equilibrium stress, but $\Gamma_1$ is not a (weighted) intrinsic Delaunay triangulation.}
\label{F:incoherent}
\end{figure}

More generally, for any positive integer $k$, let $\Gamma_k$ denote the $k\times k$ covering of $\Gamma_1$.  The vertices of $\Gamma_k$ form a regular $k\times k$ square toroidal lattice, and the edges of $\Gamma_k$ fall into three parallel families, with displacement vectors $\colvec{1/k}{0}$,  $\colvec{2/k}{1/k}$, and $\colvec{1/k}{1/k}$.  Every positive stress vector where all parallel edges have equal stress coefficients is an equilibrium stress.

For the sake of argument, suppose $\Gamma_k$ is coherent.  Consider three consecutive darts $\arc{p}{q}$, $\arc{p}{r}$, and $\arc{p}{s}$ in counterclockwise order around a common tail $p$, with displacement vectors $\colvec{1/k}{0}$, $\colvec{2/k}{1/k}$, and $\colvec{1/k}{1/k}$, respectively.  The local Delaunay determinant test \eqref{Eq:Delaunay} implies that the weights of these four vertices satisfy the inequality $\pi_p + \pi_r > \pi_q + \pi_s + 1/k^2$.
%
%
Every vertex of $\Gamma_k$ appears in exactly four inequalities of this form---twice on the left and twice on the right---so summing all~$k^2$ such inequalities and canceling equal terms yields the obvious contradiction $0>1$.
\end{proof}

\subsection{Example}
\label{SS:example}

As a running example, let $\Gamma$ be the (unweighted) intrinsic Delaunay triangulation of the seven points
\(
	\colvec{0}{0},
	\colvec{1/7}{3/7},\allowbreak
	\colvec{2/7}{6/7},
	\colvec{3/7}{2/7},
	\colvec{4/7}{5/7},
	\colvec{5/7}{1/7},
	\colvec{6/7}{4/7}
\)
on the square flat torus $\Torus_\square$, and let $\Gamma^*$ be the corresponding intrinsic Voronoi diagram, as shown in Figure~\ref{F:Delaunay}.  The triangulation~$\Gamma$ is a highly symmetric geodesic embedding of the complete graph $K_7$; torus embeddings isomorphic to $\Gamma$ and $\Gamma^*$ were studied in several early seminal works on combinatorial topology \cite{m-tpe-1886, h-mct-1890, h-upn-1891}.

\begin{figure}[ht]
\centering
\includegraphics[scale=0.5,page=1]{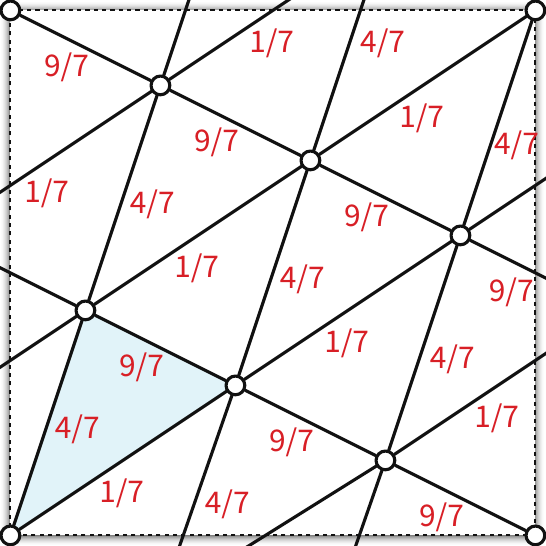}
\qquad
\qquad
\includegraphics[scale=0.5,page=2]{Fig/K7-torus-clipped}
\caption{An intrinsic Delaunay triangulation, its dual Voronoi diagram, and their induced equilibrium stresses.  Compare with Figures~\ref{F:uniform} and \ref{F:skew}.}
\label{F:Delaunay}
\end{figure}

The edges of $\Gamma$ fall into three equivalence classes, with slopes $3$, $2/3$,  $-1/2$ and lengths $\sqrt{10}/7$, $\sqrt{5}/7$, $\sqrt{14}/7$, respectively.  
The triangle $\colvec{0}{0}, \colvec{1/7}{3/7}, \colvec{3/7}{2/7}$, shaded in Figure~\ref{F:Delaunay}, has circumcenter $\colvec{19/98}{17/98}$.  Measuring slopes and distances to the nearby edge midpoints, we find that corresponding dual edges in $\Gamma^*$ have slopes $-1/3$, $-3/2$, and $2$ and lengths $4\sqrt{10}/49$, $\sqrt{5}/49$, and $9\sqrt{14}/49$, respectively.  These dual slopes confirm that $\Gamma$ and $\Gamma^*$ are reciprocal (as are any Delaunay triangulation and its dual Voronoi diagram).  The dual edge lengths imply that assigning stress coefficients $4/7$, $1/7$, and $9/7$ to the edges of $\Gamma$ yields an equilibrium stress for~$\Gamma$, and symmetrically, the stress coefficients $7/4$, $7$, and $7/9$ yield an equilibrium stress for $\Gamma^*$.

Of course, this is not the only equilibrium stress for $\Gamma$; indeed, symmetry implies that $\Gamma$ is in equilibrium with respect to the uniform stress $\omega \equiv 1$.  However, there is no reciprocal embedding $\Gamma^*$ such that every edge in $\Gamma$ has the same length as the corresponding dual edge in $\Gamma^*$.

The doubly-periodic universal cover $\Universal{\Gamma}$ is also in equilibrium with respect to the uniform stress $\omega\equiv 1$.  Thus, the classical Maxwell--Cremona correspondence implies a dual embedding $(\Universal{\Gamma})^*$ in which every dual edge is orthogonal to and has the same length as its corresponding primal edge in~$\Universal{\Gamma}$.  (Borcea and Streinu~\cite[Proposition 2]{bs-lsppf-15} discuss how to solve the infinite linear system giving the heights of the corresponding polyhedral lifting of~$\Universal{\Gamma}$.)  Symmetry implies that $(\Universal{\Gamma})^*$ is doubly-periodic.  Crucially, however,  $\Universal{\Gamma}$ and $(\Universal{\Gamma})^*$ have \emph{different period lattices}.  Specifically, the period lattice of $(\Universal{\Gamma})^*$ is generated by the vectors $\colvec{2}{-1}$ and $\colvec{-1}{2}$; see Figure \ref{F:uniform}.

\begin{figure}[ht]
\centering
\qquad
\raisebox{-0.5\height}{\includegraphics[scale=0.5,page=1]{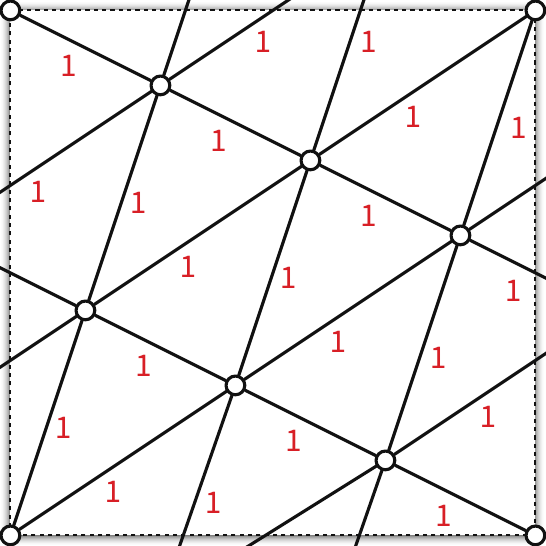}}
\qquad
\raisebox{-0.5\height}{\includegraphics[scale=0.5,page=2]{Fig/K7-torus-uniform}}
\caption{A “reciprocal” embedding (at half scale) induced by the uniform equilibrium stress~$\omega\equiv 1$.  Compare with Figures~\ref{F:Delaunay} and \ref{F:skew}.}
\label{F:uniform}
\end{figure}

Understanding which equilibrium stresses correspond to reciprocal embeddings is the topic of Section~\ref{S:reciprocal}.  In particular, in that section we describe a simple necessary and sufficient condition for an equilibrium stress to be reciprocal, which the unit stress for $\Gamma$ fails.

\section{Coherent = Reciprocal}

Unlike in the previous and following sections, the equivalence between coherent embeddings and reciprocal embeddings generalizes fully from the plane to every flat torus.  However, there is an important difference from the planar setting.  In both the plane and the torus, every translation of a reciprocal diagram is another reciprocal diagram.  For a coherent planar embedding $\Gamma$, \emph{every} reciprocal diagram is a weighted Voronoi diagram of the vertices of $\Gamma$, but \emph{exactly one} reciprocal diagram of a coherent toroidal embedding $\Gamma$ is a weighted Voronoi diagram of the vertices of $\Gamma$.  Said differently, every coherent planar embedding is a weighted Delaunay complex with respect to a three-dimensional space of vertex weights, which correspond to translations of any convex polyhedral lifting, but every coherent \emph{toroidal} embedding is a weighted Delaunay complex with respect to only a \emph{one}-dimensional space of vertex weights.

\subsection{Notation}

Throughout this section, we fix a non-singular matrix $M = \rowvec{u}{v}$ where $u,v \in \R^2$
are column vectors and $\det M > 0$; we also fix a toroidal embedding $\Gamma\colon G\into \Torus_M$.  We primarily work with the universal cover $\Universal{\Gamma}$ of $\Gamma$; if we are given a reciprocal embedding $\Gamma^*$, we also work with its universal cover $\Universal{\Gamma}^*$ (which is reciprocal to $\Universal{\Gamma}$).
Vertices of $\Universal{\Gamma}$ are denoted by the letters~$p$ and~$q$ and treated as column
vectors in~$\R^2$.
A generic face in $\Universal{\Gamma}$ is denoted by the letter~$f$; the corresponding dual vertex in~$\Universal{\Gamma}^*$ is denoted $f^*$ and interpreted as a row vector.
To avoid nested subscripts when darts are indexed, for displacement vectors we write~$\Delta_i = \Delta_{d_i}$ and
$\omega_i = \omega_{d_i}$, and therefore by Lemma~\ref{L:rec-implies-equ}, the dual displacement vectors are given by~$\Delta^*_i = \omega_i\Delta_i^\perp$.  
For any integers~$a$ and~$b$, the translation $p + au + bv$ of any vertex~$p$ of $\Universal{\Gamma}$ is another vertex of $\Universal{\Gamma}$, and the translation $f + au + bv$ of any face $f$ of~$\Universal{\Gamma}$ is another face of~$\Universal{\Gamma}$. It follows that $(f + au + bv)^* = f^* + au^T + bv^T$.

\subsection{Results}

The following lemma follows directly from the definitions of weighted Delaunay graphs and their dual weighted Voronoi diagrams; see, for example, Aurenhammer \cite{a-pdpaa-87,ai-gravd-88}.

\begin{lemma}
\label{L:coh-implies-recip}
Let $\Gamma$ be a weighted Delaunay complex on some flat torus $\Torus_M$, and let $\Gamma^*$ be the corresponding weighted Voronoi diagram on $\Torus_M$.  Every edge $e$ of $\Gamma$ is orthogonal to its dual~$e^*$.  In short, every coherent toroidal embedding is reciprocal.
\end{lemma}


The converse of this lemma is false; unlike in the plane, a reciprocal diagram $\Gamma^*$ for a fixed weighted Delaunay complex~$\Gamma$ is \emph{not} necessarily a weighted Voronoi diagram of the vertices of $\Gamma$.  Rather, as we describe below, a unique translation of~$\Gamma^*$ is such a weighted Voronoi diagram.

Fix a toroidal geodesic embedding $\Gamma\colon G\to \Torus_M$.  Borcea and Streinu’s extension of Maxwell's theorem to periodic planar frameworks \cite{bs-lsppf-15} implies a (non-unique) convex polyhedral lifting $z \colon \R^2 \to \R$ of the universal cover~$\Universal{\Gamma}$ of~$\Gamma$, where the gradient vector $\smash{\nabla z|_f}$ within any face $f$ is equal to the coordinate vector of the dual vertex $f^*$ in a planar reciprocal embedding $\Universal{\Gamma}^*$.
%
To make this lifting unique, we fix a vertex~$o$ of $\Universal{\Gamma}$ to lie at the origin $\colvec{0}{0}$, and we require $z(o) = 0$.

Define the weight of each vertex $p \in \Universal{\Gamma}$ as
\[
\pi_p = \Half\norm{p}^2 - z(p).
\]
By definition, $\pi_o = 0$.  The determinant conditions \eqref{Eq:Delaunay} and \eqref{Eq:Delaunay2} for an edge~$e$ to be locally Delaunay with respect to these vertex weights  $\pi_p$ are both equivalent to the restriction of the lifting $z$ to the faces incident to $e$ being convex.
Because $z$ is a convex polyhedral lifting, these
weights establish that $\Universal{\Gamma}$ is the intrinsic weighted Delaunay graph of its vertex set.

Translating the universal cover $\Universal{\Gamma}^*$ of the reciprocal graph $\Gamma^*$ adds a global linear term to the lifting function $z$, and therefore to the Delaunay weights~$\pi_p$.  The main result of this section is that there is a unique translation such that the corresponding Delaunay weights~$\pi_p$ are periodic with respect to the lattice generated by the columns of $M$.

To compute $z(q)$ for any point $q \in \R^2$, we choose an arbitrary face $f$ of $\Universal{\Gamma}$ that contains $q$ and
identify the equation $z|_f(q) = \eta q + c$ of the plane through the lift of $f$, where $\eta \in\Real^2$ is a row vector and $c \in \R$. Borcea and Streinu~\cite{bs-lsppf-15} give a calculation for
$\eta$ and $c$, which for our setting can be written as follows:

\begin{lemma}[{Borcea and Streinu\cite[Eq.~7]{bs-lsppf-15}}]
\label{L:bs-maxwell}
Let $\Gamma$ and $\Gamma^*$ be reciprocal toroidal embeddings, and let~$\Universal{\Gamma}$ and $\Universal{\Gamma}^*$ be respective universal covers of $\Gamma$ and $\Gamma^*$, with positive equilibrium stress $\omega$ on $\Universal{\Gamma}$ and lifting of $\Universal{\Gamma}$ as induced by Maxwell's theorem.  Fix the coordinates of a vertex $o$ of $\Universal{\Gamma}$ to $\colvec00$, and fix $z(o) = 0$.
Then for any face $f$ of $\Universal{\Gamma}$, the function~$z|_f$ can be computed as follows:
	\begin{itemize}
	\item Pick an arbitrary \textbf{root} face $f_0$ incident to $o$.
	\item Pick an arbitrary path from $f_0^*$ to $f^*$ in $\Universal{\Gamma}^*$,
		and let $d_1^*,\ldots,d_\ell^*$	be the dual darts along this path.
		By definition, we have $f^* = f_0^* + \sum_{i=1}^\ell \Delta^*_i$.
		Set $C(f) = \sum_{i=1}^\ell \omega_i \rowdet{p_i}{q_i}$, where
		$d_i = p_i \arcto q_i$.
	\item Then $z|_f(q) = f^*q + C(f)$ for all $q\in f$. In particular, $C(f)$ is the height of the intersection of this plane with the $z$-axis.
	\end{itemize}
Reciprocality of $\Universal{\Gamma}$ and $\Universal{\Gamma}^*$ implies that the actual choice of root face $f_0^*$ and the path to~$f^*$ do not matter.
\end{lemma}

Fix two reciprocal embeddings $\Gamma$ and $\Gamma^*$ on $\Torus_M$ where $M = \rowvec{u}{v}$.  We use this explicit computation to establish the existence of a translation of~$\Gamma^*$ such that $\pi_o = \pi_{o+u} = \pi_{o+v} = 0$. We then show that after this translation, every lift of the same vertex of $\Gamma$ has the same Delaunay weight.

\begin{lemma}
\label{L:translation}
There is a unique translation of $\Universal{\Gamma}^*$ such that $\pi_o = \pi_{o+u} = \pi_{o+v} = 0$.
Specifically, this translation places the dual vertex of the root face $f_0$ at the point
\[
	f_0^* = \left(-\Half\rowvec{\norm{u}^2}{\norm{v}^2} -
					\rowvec{C(f_0+u)}{C(f_0+v)}\right) M^{-1}.
\]
\end{lemma}

\begin{proof}
Lemma \ref{L:bs-maxwell} implies that
\[
	z(u)
	~=~
	(f_0 + u)^* u + C(f_0 + u)
	~=~
	f_0^* u + \norm{u}^2 + C(f_0 + u),
\]
and by definition, $\pi_{o+u} = 0$ if and only if $z(u) = \Half\norm{u}^2$.  Thus, $\pi_{o+u}=0$ if and only if $f_0^* u = -\Half\norm{u}^2 - C(f_0 + u)$.  A symmetric argument implies $\pi_{o+v}=0$ if and only if $f_0^* v = -\Half\norm{v}^2 - C(f_0 + v)$.
\end{proof}

\begin{lemma}
\label{L:periodic}
If $\pi_o = \pi_{o+u} = \pi_{o+v} = 0$, then $\pi_p = \pi_{p + u} = \pi_{p + v}$ for each vertex $p$ of $\Universal{\Gamma}$. In other words, all lifts of any vertex of $\Gamma$ have equal weight.
\end{lemma}

\begin{proof}
Let $f$ be any face incident to $p$, and let $P = d^*_1,\ldots,d^*_\ell$ be an arbitrary path from~$f^*_0$ to $f^*$ in~$\Universal{\Gamma}^*$.  We compute $C(f+u)$ by traversing an arbitrary path from $f^*_0$ to~$(f_0 + u)^* = f^*_0+u^T$ followed by the translated path~$P+u$ from~$f^*_0+u^T$ to~$f^*+u^T$.  Thus by Lemma~\ref{L:bs-maxwell},~$C(f + u) = C(f_0 + u) + \sum_{i=1}^\ell \omega_i\rowdet{(p_i+u)}{(q_i+u)}$, and $f^* = f_0^* + \sum_{i=1}^\ell \Delta^*_i$. We thus have
\begin{align*}
C(f + u)
	&= C(f_0 + u) + \sum_{i=1}^\ell \omega_i\rowdet{(p_i+u)}{(q_i+u)} &\\
	&= C(f_0 + u) + \sum_{i=1}^\ell \omega_i\rowdet{p_i}{q_i} - \sum_{i=1}^\ell \Delta^*_i u \\
	&= C(f_0 + u) + C(f) - \sum_{i=1}^\ell \Delta^*_iu \\
	&= -\Half\norm{u}^2 - f_0^* u + C(f) - \sum_{i=1}^\ell \Delta^*_iu \\
	&= -\Half\norm{u}^2 - f^* u + C(f).
\end{align*}
It follows that
\begin{align*}
\pi_{p + u}
	&= \Half\norm{p + u}^2 - z(p + u) \\
	&= \Half\norm{p + u}^2 - \left(C(f + u) + (f^* + u^T)(p + u)\right)\\
	&= \Half\norm{p + u}^2 - \left(-\Half\norm{u}^2 - f^* u + C(f) + f^* p + f^* u + u^T p + \norm{u}^2\right) \\
	&= \Half\norm{p + u}^2 - z(p) - \Half\norm{u}^2 - u^Tp \\
	&= \Half\norm{p}^2 + \Half\norm{u}^2 + u^Tp - z(p) - \Half\norm{u}^2 - u^Tp \\
	&= \Half\norm{p}^2 - z(p) \\
	&= \pi_p.
\end{align*}
A similar computation implies $\pi_{p+v} = \pi_p$.
\end{proof}

Projecting from the universal cover back to the torus, we obtain weights for the vertices of $\Gamma$, with respect to which $\Gamma$ is an intrinsic weighted Delaunay complex, and a unique translation of~$\Gamma^*$ that is the corresponding intrinsic weighted Voronoi diagram.  Moreover, these Delaunay vertex weights become unique if we fix the weight of an arbitrary vertex to be $0$.

\begin{theorem}
\label{T:recip-implies-del}
Let $\Gamma$ and $\Gamma^*$ be reciprocal geodesic embeddings on some flat torus $\Torus_M$.  $\Gamma$ is a weighted Delaunay complex, and a unique translation of~$\Gamma^*$ is the corresponding weighted Voronoi diagram.  In short, every reciprocal toroidal embedding is coherent.
\end{theorem}

\section{Equilibrium Implies Reciprocal, Sort Of}
\label{S:reciprocal}


Now fix an essentially simple, essentially 3-connected geodesic embedding $\Gamma$ on the \emph{square} flat torus~$\Torus_\square$, along with a positive equilibrium stress $\omega$ for $\Gamma$.  In this section, we describe simple necessary and sufficient conditions for $\omega$ to be a reciprocal stress for $\Gamma$.  More generally, we show that when $\omega$ is a \emph{positive} equilibrium stress, there is an essentially unique flat torus~$\Torus_M$ such that a unique scalar multiple of $\omega$ is a reciprocal stress for the image of $\Gamma$ on~$\Torus_M$.


\subsection{Cocirculations and Cohomology}

Fix arbitrary dual geodesic embeddings $\Gamma\colon G\to\Torus$ and $\Gamma^*\colon G^*\to\Torus$.  A \EMPH{cocirculation} in~$\Gamma$ is a \emph{row} vector $\theta\in\Real^E$ whose transpose describes a circulation in $G^*$.  The \EMPH{cohomology} class $[\theta]^* = [\theta]^*_\Gamma$ of any cocirculation (with respect to $\Gamma$) is the transpose of the homology class of the circulation~$\theta^T$ in~$\Gamma^*$.  Recall that $\Lambda_\Gamma$ is the $2\times E$ matrix whose columns are homology vectors of edges in~$\Gamma$.  Let $\lambda_1$ and~$\lambda_2$ denote the first and second rows of~$\Lambda_\Gamma$.  

\begin{lemma}
\label{L:cocirc}
Fix a geodesic embedding $\Gamma\colon G\to \Torus$.  The row vectors $\lambda_1$ and $\lambda_2$ describe cocirculations in $\Gamma$ with cohomology classes $[\lambda_1]^* = \rowvec{0}{1}$ and $[\lambda_2]^* = \rowvec{-1}{0}$.
\end{lemma}

\begin{proof}
Without loss of generality, assume that $\Torus = \Torus_\square$ and no vertices of $\Gamma$ lie on the boundary of the fundamental square $\square$.  Let $\gamma_1$ and $\gamma_2$ denote directed cycles in~$\Torus_\square$ (\emph{not} in~$\Gamma$) induced by the boundary edges of $\square$, directed respectively rightward and upward.

Let $d_0, d_1, \dots, d_{k-1}$ be the sequence of darts in $\Gamma$ whose images in $\Gamma$ cross $\gamma_2$ from left to right, indexed by the upward order of their intersection points.  Each dart $d$ that appears in this sequence appears exactly $\lambda_1(d)$ times, once for each crossing.  For each index $i$, we have~$\Left(d_i) = \Right(d_{i+1\bmod k})$; thus, the corresponding sequence of dual darts~$d^*_0, d^*_1, \dots, d^*_{k-1}$ describes a closed walk in $\Gamma^*$.  This closed walk can be continuously deformed to $\gamma_2$, so it has the same homology class as $\gamma_2$; see Figure~\ref{F:cocirc}.  We conclude that~$[\lambda_1]^* = \rowvec{0}{1}$.  

\begin{figure}[ht]
\centering
\includegraphics[scale=0.41,page=1]{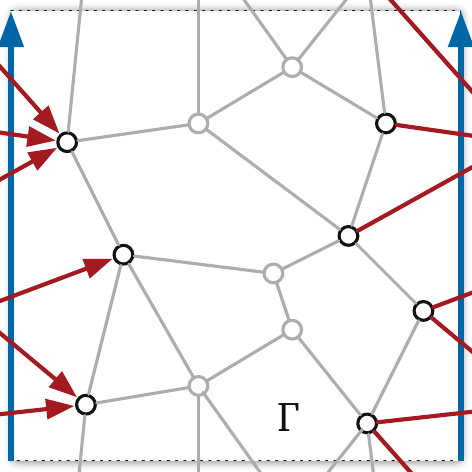}\quad
\includegraphics[scale=0.41,page=2]{Fig/cohomology}\qquad\quad
\includegraphics[scale=0.41,page=1]{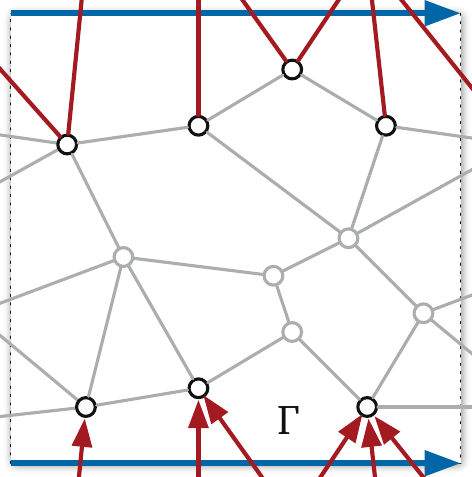}\quad
\includegraphics[scale=0.41,page=2]{Fig/cohomologx}
\caption{Proof of Lemma \ref{L:cocirc}: The darts in $\Gamma$ crossing either boundary edge of the fundamental square dualize to a closed walk in~$\Gamma^*$ parallel to that boundary edge.}
\label{F:cocirc}
\end{figure}

Symmetrically, the darts crossing $\gamma_1$ upward define a closed walk in $\Gamma^*$ in the same homology class as the reversal of $\gamma_1$, and therefore $[\lambda_2]^* = \rowvec{-1}{0}$.
\end{proof}


\subsection{The Square Flat Torus}

Before considering arbitrary flat tori, we first establish necessary and sufficient conditions for reciprocal stresses for embeddings on the \emph{square} flat torus.  Fix an essentially simple, essentially 3-connected embedding $\Gamma\colon G\to \Torus_\square$ and let $\omega$ be a positive equilibrium stress on~$\Gamma$. Let $\Delta$ be the $2 \times E$ displacement matrix of $\Gamma$, and let $\Omega$ be the $E\times E$ matrix whose diagonal entries are $\Omega_{e,e} = \omega_e$ and whose off-diagonal entries are all~$0$. Our results will be in terms of the \EMPH{covariance} matrix $\Delta\Omega\Delta^T$. Recall that $A^\perp = (JA)^T$ and that $\abs{e}$ denotes the length of an edge $e$ in $\Gamma$.

\begin{lemma}
\label{L:square-abg}
If $\omega$ is a reciprocal stress for a geodesic embedding $\Gamma$ on $\Torus_\square$, then $\Delta\Omega\Delta^T = \begin{psmallmatrix} 1 & 0 \\ 0 & 1 \end{psmallmatrix}$.
\end{lemma}

\begin{proof}
Suppose $\omega$ is a reciprocal stress for $\Gamma$.  Then by definition, there is a geodesic embedding $\Gamma^*\colon {G^*\to \Torus_\square}$ such that $e\perp e^*$ and 
$\abs{e^*} = \omega_e\abs{e}$ for every edge $e$ of $\Gamma$.  Let 
$\Delta^* = (\Delta\Omega)^\perp$ denote the $E \times 2$ matrix whose rows are the
displacement row vectors of $\Gamma^*$.

Recall from Lemma~\ref{L:cocirc} that the first and second rows of $\Lambda_\Gamma$ describe cocirculations of $\Gamma$ with cohomology classes $\rowvec{0}{1}$ and $\rowvec{-1}{0}$, respectively.  Applying Lemma~\ref{L:harmonic} to $\Gamma^*$ implies~$\theta\Delta^* = [\theta]^*$ for any cocirculation $\theta$ in $\Gamma$.  It follows immediately that
$\Lambda_\Gamma\Delta^* = \begin{psmallmatrix} 0 & 1 \\ -1 & 0 \end{psmallmatrix} = -J$.

Because the rows of $\Delta^*$ are the displacement vectors of $\Gamma^*$, for every vertex $p$ of $\Gamma$ we have
\begin{equation}
	\sum_{q\colon pq\in E} \Delta^*_{(\arc{p}{q})^*}
	~=~
	\sum_{d \colon \Tail(d) = p} \Delta^*_{d^*}
	~=~
	\sum_{d \colon \Left(d^*) = p^*} \Delta^*_{d^*}
	~=~
	\rowvec{0}{0}.
	\label{Eq:circulation}
\end{equation}
It follows that the \emph{columns} of $\Delta^*$ describe circulations in $\Gamma$.
Lemma~\ref{L:harmonic} now implies that~$\Delta\Delta^* = -J$.  We conclude that $\Delta\Omega\Delta^T = \Delta\Delta^* J = \begin{psmallmatrix} 1 & 0 \\ 0 & 1 \end{psmallmatrix}$.
\end{proof}

\begin{lemma}
\label{L:dual-embed-on-square}
Fix an essentially simple, essentially 3-connected geodesic embedding $\Gamma\colon G \to \Torus_\square$ and an $E \times 2$ matrix $\Delta^*$.  If $\Lambda_\Gamma\Delta^* = -J$, then $\Delta^*$ is the displacement matrix of a geodesic drawing $\Gamma^*$ on $\Torus_\square$ that is homotopic to a geometric dual of $\Gamma$.  Moreover, if that drawing admits a positive equilibrium stress, it is actually an embedding.
\end{lemma}

\begin{proof}
Let $\lambda_1$ and $\lambda_2$ denote the rows of $\Lambda_\Gamma$.  Rewriting the identity
$\Lambda_\Gamma\Delta^* = -J$ in terms of these row vectors gives us
\(
	\sum_e \Delta^*_e\lambda_{1,e} = \rowvec{0}{1} = [\lambda_1]^*
\)
\text{and} 
\(
	\sum_e \Delta^*_e\lambda_{2,e} = {\rowvec{-1}{0} = [\lambda_2]^*}.
\)
Extending by linearity, we have $\sum_e \Delta^*_e\theta_e = [\theta]^*$ for every cocirculation~$\theta$ in~$\Gamma^*$.  The result now follows from Corollary~\ref{C:embed-on-square}. 
\end{proof}

\begin{lemma}
\label{L:abg-square}
Fix an essentially simple, essentially 3-connected geodesic embedding $\Gamma\colon G \to \Torus_\square$ with a positive equilibrium stress $\omega$. 
If $\Delta\Omega\Delta^T = \begin{psmallmatrix} 1 & 0 \\ 0 & 1 \end{psmallmatrix}$, then $\omega$ is a reciprocal stress for~$\Gamma$.
\end{lemma}

\begin{proof}
Set $\Delta^* = (\Delta\Omega)^\perp$.  Because $\omega$ is an equilibrium stress for $\Gamma$, for every vertex $p$ of $\Gamma$ we have
\begin{equation}
	\sum_{q\colon pq\in E} \Delta^*_{(\arc{p}{q})^*}
	~=~
	\sum_{q\colon pq\in E} \omega_{pq} \Delta_{\arc{p}{q}}^\perp
	~=~
	\rowvec{0}{0}.
	\label{Eq:circulation2}
\end{equation}
It follows that the columns of $\Delta^*$ describe circulations in~$\Gamma$, and therefore Lemma~\ref{L:harmonic} implies~$\Lambda_\Gamma\Delta^* = \Delta\Delta^* = \Delta (\Delta\Omega)^\perp = \Delta\Omega\Delta^T J^T = -J$.

Lemma \ref{L:dual-embed-on-square} now implies that~$\Delta^*$ is the displacement matrix of a drawing~$\Gamma^*$ homotopic to a dual to $\Gamma$.  Moreover, the stress vector $\omega^*$ defined by $\omega^*_{e^*} = 1/\omega_e$ is an equilibrium stress for $\Gamma^*$: under this stress vector, the darts leaving any dual vertex~$f^*$ are dual to the clockwise boundary cycle of face $f$ in $\Gamma$.  Thus since $\omega^*$ is positive, $\Gamma^*$ is in fact a dual embedding.  
By construction, each edge of $\Gamma^*$ is orthogonal to the corresponding edge of~$\Gamma$.
\end{proof}

\subsection{Force Diagrams}
\label{SS:force}

The results of the previous section have an alternative interpretation that may be more intuitive.  Let~$\Gamma$ be any geodesic embedding on the unit square flat torus $\Torus_\square$.  Recall from Section \ref{SS:example} that any equilibrium stress~$\omega$ on $\Gamma$ induces an equilibrium stress on its universal cover $\Universal{\Gamma}$, which in turn induces a reciprocal diagram~$(\Universal{\Gamma})^*$, which is unique up to translation, by the classical Maxwell--Cremona correspondence.  This infinite plane graph $(\Universal{\Gamma})^*$ is doubly-periodic, but in general with a different period lattice from the universal cover $\Universal{\Gamma}$.

Said differently, we can always construct another geodesic torus embedding $\Digamma$ that is combinatorially dual to $\Gamma$, such that for every edge $e$ of $\Gamma$, the corresponding edge $e^*$ of~$\Digamma$ is “orthogonal” to $e$ (insofar as their displacement vectors, considered as vectors in $\R^2$, are orthogonal) and has length $\omega_e\cdot\abs{e}$; however, this embedding $\Digamma$ does not necessarily lie on the \emph{square} flat torus.  Specifically, $\Digamma$ is the quotient of some reciprocal diagram $(\Universal{\Gamma})^*$ of the universal cover with respect to its period lattice.  We call~$\Digamma$ a \EMPH{force diagram} of $\Gamma$ with respect to $\omega$.  Force diagrams are unique up to translation.  A force diagram $\Digamma$ lies on the same flat torus $\Torus_\square$ as~$\Gamma$ if and only if $\omega$ is a reciprocal stress for $\Gamma$.

\begin{lemma}
\label{L:force}
Fix an essentially simple, essentially 3-connected geodesic embedding $\Gamma\colon G \to \Torus_\square$, and let~$\omega$ be a positive equilibrium stress for $\Gamma$.  Every force diagram of $\Gamma$ with respect to $\omega$ lies on the flat torus~$\Torus_M$, where $M = J\Delta\Omega\Delta^T J^T$.
\end{lemma}

\begin{proof}
As usual, let $\Delta$ be the displacement matrix of $\Gamma$.  Let $\Delta^*$ denote the displacement matrix of some force diagram $\Digamma$; by definition, we have $\Delta^* = (\Delta\Omega)^\perp = \Omega\Delta^T J^T$.   Equation~\eqref{Eq:circulation2} implies that the columns of $\Delta^*$ are circulations in $\Gamma$.  Thus, Lemma~\ref{L:harmonic} implies that $\Lambda\Delta^* = \Delta\Delta^* = \Delta\Omega\Delta^TJ^T$.

Set $M = J\Delta\Delta^* = J\Delta\Omega\Delta^TJ^T$.  We immediately have $\Lambda\Delta^* = J^{-1}M = -J M = -J M^T$ and therefore $\Lambda\Delta^*(M^T)^{-1} = -J$.  Lemma~\ref{L:dual-embed-on-square} implies that $\Delta^*(M^T)^{-1}$ is the displacement matrix of a geodesic embedding $\Gamma^* \colon G^*\to \Torus_\square$ dual to $\Gamma$.  It follows that $\Delta^*$ is the displacement matrix of the affine image of $\Gamma^*$ on~$\Torus_M$.  We conclude that the force diagram~$\Digamma$ is a translation of $\underline{M}\circ \Gamma^*$.
\end{proof}

\subsection{Arbitrary Flat Tori}

Now we generalize our earlier analysis to embeddings on arbitrary flat tori.

Fix a \emph{reference} embedding $\Gamma$ on the square flat torus $\Torus_\square$, and let $\omega$ be a positive equilibrium stress on $\Gamma$.  Let $\Delta$ be the displacement matrix of $\Gamma$, and $\Omega$ be the matrix whose diagonal entries are $\Omega_{e,e} = \omega_e$ and whose off-diagonal entries are all $0$.  

Now fix a non-singular $2\times 2$ matrix $M$.  Recall from Lemma~\ref{L:equ-affine} that $\omega$ is also a positive equilibrium stress for the induced embedding $\underline{M} \circ \Gamma\colon G\to \Torus_M$.  In this section, we establish necessary and sufficient conditions for $\omega$ to be a \emph{reciprocal} stress on $\underline{M} \circ \Gamma$.  We state these conditions in terms of the (entries of the) covariance matrix $\Delta\Omega\Delta^T = \begin{psmallmatrix} \alpha & \gamma \\ \gamma & \beta \end{psmallmatrix}$, where
\begin{equation}
	\alpha = \sum_e \omega_e \Dx_e^2,
	\qquad \beta = \sum_e \omega_e \Dy_e^2,
	\qquad \gamma = \sum_e \omega_e \Dx_e\Dy_e.
	\label{Eq:alphabet}
\end{equation}
We emphasize that these covariance parameters are defined with respect to the \emph{reference} embedding~$\Gamma$.

\begin{lemma}
\label{L:nonsq-abg}
If $\omega$ is a reciprocal stress for the geodesic embedding $\underline{M} \circ \Gamma\colon G \to \Torus_M$, then~$\alpha\beta-\gamma^2 = 1$; in particular, if $M = \begin{psmallmatrix} a & b \\ c & d \end{psmallmatrix}$, then
\[
	\alpha = \frac{b^2+d^2}{ad-bc},
	\qquad
	\beta = \frac{a^2+c^2}{ad-bc},
	\qquad
	\gamma = \frac{-(ab+cd)}{ad-bc}.
\]
For example, if $M = \rowvec{u}{v}$ where $u,v \in \R^2$ are column vectors and $\det M = 1$, then~$\Delta\Omega\Delta^T = \begin{psmallmatrix} v\cdot v & -u \cdot v \\[0.25ex] -u \cdot v & u\cdot u \end{psmallmatrix}$.
\end{lemma}

\begin{proof}
Suppose $\omega$ is a reciprocal stress for $\underline{M} \circ \Gamma$.  Then by definition there is a geodesic embedding $\underline{M} \circ \Gamma^*\colon G^* \to \Torus_M$ dual to $\underline{M} \circ \Gamma$ such that for every edge $e$, the geodesic segments~$\underline{M}(\Gamma(e))$ and $\underline{M}(\Gamma^*(e^*))$ are orthogonal, and the ratio of their lengths is $\omega_e$.  (The reference embeddings $\Gamma$ and~$\Gamma^*$ on $\Torus_\square$ are also duals, but they are not reciprocal unless the matrix $M$ is orthogonal.)

Let~$\Delta$ denote the $2\times E$ displacement matrix of the reference embedding $\Gamma$, whose columns are the displacement vectors of~$\Gamma$.  The columns of $M\Delta$ are the displacement vectors for~$\underline{M} \circ \Gamma$.  Thus, the displacement row vectors of~$\underline{M} \circ \Gamma^*$ are given by the rows of the $E \times 2$ matrix $(M\Delta\Omega)^\perp$.  Finally, let $\Delta^* = (M\Delta\Omega)^\perp (M^T)^{-1}$ denote the displacement row vectors for the dual reference embedding $\Gamma^*$.  
We can rewrite this definition as
\begin{equation}
\begin{aligned}
\Delta^*
	&= (M\Delta\Omega)^\perp(M^T)^{-1}	
			\\
	&= (M\Delta\Omega)^\perp(M^{-1})^T
			\\
	&= (JM\Delta\Omega)^T(M^{-1})^T
			\\
	&= \Omega \Delta^T\, M^T J^T (M^{-1})^T,
\end{aligned}
\label{Eq:Delta-star}
\end{equation}
which, since $J^T = J^{-1}$, implies that $\Omega\Delta^T = \Delta^*(M^T J^T (M^{-1})^T)^{-1} = \Delta^* M^T J (M^{-1})^T$.
%

Because the rows of $\Delta^*$ are the displacement vectors for $\Gamma^*$, equation \eqref{Eq:circulation} implies that the \emph{columns} of $\Delta^*$ describe circulations in $\Gamma$, and therefore $\Delta\Delta^* = \Lambda\Delta^* = \begin{psmallmatrix} 0 & 1 \\ -1 & 0 \end{psmallmatrix} = -J$ by Lemmas~\ref{L:harmonic} and~\ref{L:cocirc}, as explained in the second paragraph of the proof of Lemma \ref{L:square-abg}.  We conclude that
\begin{align*}
	\Delta\Omega\Delta^T
	&= \Delta\Delta^* M^T J (M^{-1})^T
		= J^T M^T J (M^{-1})^T
		\\
	&= \frac{1}{ad-bc}
			\begin{pmatrix} 0 & 1 \\ -1 & 0 \end{pmatrix}
			\begin{pmatrix} a & c \\ b & d \end{pmatrix}
			\begin{pmatrix} 0 & -1 \\ 1 & 0 \end{pmatrix}
			\begin{pmatrix} d & -c \\ -b & a \end{pmatrix}
		\\
	&= \frac{1}{ad-bc}
			\qquad \begin{pmatrix} b & d \\ -a & -c \end{pmatrix}
			\qquad\quad \begin{pmatrix} b & -a \\ d & -c \end{pmatrix}
		\\
	&= \frac{1}{ad-bc}
			\qquad\quad\begin{pmatrix} b^2+d^2 & -ab-cd \\ -ab-cd & a^2+c^2 \end{pmatrix}.
\end{align*}
Routine calculation now implies that $\alpha\beta - \gamma^2 = \det \Delta\Omega\Delta^T = 1$.
\end{proof}

\begin{corollary}
If $\omega$ is a reciprocal stress for a geodesic embedding $\underline{M} \circ \Gamma\colon G \to \Torus_M$, then~$M = \sigma R\begin{psmallmatrix} \beta & -\gamma \\ 0 & 1 \end{psmallmatrix}$ for some $2\times 2$ rotation matrix~$R$ and some real number $\sigma>0$.
\end{corollary}

\begin{proof}
Reciprocality is preserved by rotating and scaling the fundamental parallelogram~$\lozenge_M$, so it suffices to consider the special case $M = \begin{psmallmatrix} a & b \\ 0 & 1 \end{psmallmatrix}$.  In this special case, Lemma~\ref{L:nonsq-abg} immediately implies $\beta = a$ and $\gamma = -b$.
\end{proof}

\begin{lemma}
\label{L:abg-nonsq}
Fix an essentially simple, essentially 3-connected geodesic embedding $\Gamma\colon G \to \Torus_\square$ with a positive equilibrium stress $\omega$.  Let $\alpha$, $\beta$, and~$\gamma$ be defined as in~Equation \eqref{Eq:alphabet}.  If $\alpha\beta - \gamma^2 = 1$, then $\omega$ is a reciprocal stress for the embedding $\underline{M} \circ \Gamma\colon G\to \Torus_M$ induced by  the matrix $M = \sigma R\begin{psmallmatrix} \beta & -\gamma \\ 0 & 1 \end{psmallmatrix}$, for any $2\times 2$ rotation matrix~$R$ and any real number $\sigma>0$.
\end{lemma}

\begin{proof}
Suppose $\alpha\beta-\gamma^2 = 1$.  Fix an arbitrary $2\times 2$ rotation matrix $R$ and an arbitrary real number $\sigma>0$, and let $M = \sigma R\begin{psmallmatrix} \beta & -\gamma \\ 0 & 1\end{psmallmatrix}$.  Let~$\Delta$ denote the $2 \times E$ \emph{reference} displacement matrix for $\Gamma$ on the square flat torus~$\Torus_\square$, and define the $E \times 2$ matrix $\Delta^* = (M\Delta\Omega)^\perp (M^T)^{-1}$.

Derivation \eqref{Eq:Delta-star} in the proof of Lemma \ref{L:nonsq-abg} implies $\Delta^* = \Omega \Delta^T (M^{-1} J M)^T$.  We easily observe that $(\sigma R)^{-1}J(\sigma R) = J$, and therefore
\begin{align*}
	M^{-1} J M
	&~=~ \begin{pmatrix} \beta & -\gamma \\ 0 & 1 \end{pmatrix}^{-1}
		\begin{pmatrix} 0 & -1 \\ 1 & 0 \end{pmatrix}
		\begin{pmatrix} \beta & -\gamma \\ 0 & 1 \end{pmatrix}
		\\
	&~=~ \frac{1}{\beta}
		\begin{pmatrix} 1 & \gamma \\ 0 & \beta\end{pmatrix}
		\begin{pmatrix} 0 & -1 \\ 1 & 0 \end{pmatrix}
		\begin{pmatrix} \beta & -\gamma \\ 0 & 1 \end{pmatrix}
		\\
	&~=~ \frac{1}{\beta}
		\begin{pmatrix} \beta\gamma & -1-\gamma^2 \\ \beta^2 & -\beta\gamma \end{pmatrix}
		\\
	&~=~
		\begin{pmatrix} \gamma & -\alpha \\ \beta & - \gamma \end{pmatrix}.
\end{align*}
It follows that
\[
	\Delta\Delta^*
	~=~ \Delta \Omega \Delta^T (M^{-1} J M)^T
	~=~ \begin{pmatrix} \alpha & \gamma \\ \gamma & \beta \end{pmatrix}
	  \begin{pmatrix} \gamma & \beta \\ -\alpha & -\gamma \end{pmatrix}
	~=~ \begin{pmatrix} 0 & \alpha\beta-\gamma^2 \\ \gamma^2-\alpha\beta & 0 \end{pmatrix}
	~=~ -J.
\]
Because $\omega$ is an equilibrium stress on $\Gamma$, for every vertex $p$ of $\Gamma$ we have
\begin{equation}
	\sum_{q\colon pq\in E} \Delta^*_{(\arc{p}{q})^*}
	~=~
	\sum_{q\colon pq\in E} \omega_{pq} \Delta_{\arc{p}{q}}^\perp (M^{-1}JMJ^T)^T
	~=~
	\rowvec{0}{0} (M^{-1}JMJ^T)^T
	~=~
	\rowvec{0}{0}.
	\label{Eq:circulation3}
\end{equation}
Once again, the columns of $\Delta^*$ describe circulations in $\Gamma$, so Lemma~\ref{L:harmonic} implies $\Lambda\Delta^* = \Delta\Delta^* = -J$.  Lemma~\ref{L:dual-embed-on-square} now implies that $\Delta^*$ is the displacement matrix of a embedding~$\Gamma^*\colon G^*\to \Torus_\square$ that is dual to~$\Gamma$.  It follows that 
$(M\Delta\Omega)^\perp = \Delta^*M^T$ is the displacement matrix of $\underline{M}\circ \Gamma^*$.  By construction, each edge of~$\underline{M} \circ \Gamma^*$ is orthogonal to its corresponding edge of~$\underline{M} \circ \Gamma$.  We conclude that $\omega$ is indeed a reciprocal stress for $\underline{M} \circ \Gamma$.
\end{proof}

Our main theorem now follows immediately. 

\begin{theorem}
\label{T:skew-rec-nonlinear}
Fix an essentially simple, essentially 3-connected geodesic embedding $\Gamma\colon G \to \Torus_\square$ with positive equilibrium stress $\omega$.  Let $\alpha$, $\beta$, and~$\gamma$ be defined as in~Equation \eqref{Eq:alphabet}.
If~$\alpha\beta - \gamma^2 = 1$, then $\omega$ is a reciprocal stress for the embedding $\underline{M} \circ \Gamma\colon G\to\Torus_M$ if and only if~$M = \sigma R \begin{psmallmatrix} \beta & -\gamma \\ 0 & 1 \end{psmallmatrix}$ for any rotation matrix~$R$ and any real number $\sigma>0$.  On the other hand, if $\alpha\beta - \gamma^2 \ne 1$, then $\omega$ is not a reciprocal stress for the affine image of $\Gamma$ on any flat torus. 
\end{theorem}

When $\omega$ is positive, \(
	\alpha\beta-\gamma^2
	=
	\frac{1}{2}\sum_{e,e'} \omega_e \omega_{e'}
			\big|
				\begin{smallmatrix}
				\Dx_{e\;} & \Dy_{e\;} \\
				\Dx_{e'} & \Dy_{e'}
				\end{smallmatrix}
			\big|^2
	> 0
\), so in fact the requirement $\alpha\beta - \gamma^2 = 1$ is just a scaling condition: Given any positive equilibrium stress $\omega$, the scaled equilibrium stress $\omega/\sqrt{\alpha\beta-\gamma^2}$ satisfies the requirement. In short, \emph{every} positive equilibrium embedding on any flat torus has a coherent affine image on \emph{some} essentially unique flat torus.

The results of this section can be reinterpreted in terms of force diagrams as follows:

\begin{lemma}
\label{L:forceM}
Fix an essentially simple, essentially 3-connected geodesic embedding $\Gamma\colon G \to \Torus_\square$, and let $\omega$ be an equilibrium stress for $\Gamma$.  For any non-singular matrix $M$, every force diagram of $\underline{M} \circ \Gamma$ with respect to $\omega$ lies on the flat torus $\Torus_N$, where $N = J M \Delta\Omega\Delta^T J^T$.
\end{lemma}

\begin{proof}
We argue exactly as in the proof of Lemma \ref{L:force}.
Let $\Digamma$ be any force diagram of~$\underline{M} \circ \Gamma$ with respect to $\omega$, and let $\Delta$ be the displacement matrix of the reference embedding~$\Gamma$.  Then the displacement matrix of~$\Digamma$ is $\Delta^* = (M\Delta\Omega)^\perp = \Omega\Delta^TM^TJ^T$.  Equation~\eqref{Eq:circulation3} and Lemma~\ref{L:harmonic} imply that $\Lambda\Delta^* = \Delta\Omega\Delta^TM^TJ^T$.

Now let $N = J M \Delta\Omega\Delta^T J^T$.  We immediately have $J^{-1} N^T = \Lambda\Delta^*$ and therefore $\Lambda\Delta^*(N^T)^{-1} = J^{-1} = -J$.
Lemma~\ref{L:dual-embed-on-square} implies that $\Delta^*(N^T)^{-1}$ is the displacement matrix of a geodesic embedding $\Gamma^* \colon G^* \to \Torus_\square$ dual to $\Gamma$.  It follows that $\Delta^*$ is the displacement matrix of the affine image of $\Gamma^*$ on $\Torus_N$; in other words, our force diagram $\Digamma$ is a translation of $\underline{N} \circ \Gamma^*$.
\end{proof}

\subsection{Example}

Let us revisit the example embedding $\Gamma$ from Section \ref{SS:example}: the symmetric embedding of $K_7$ on the square flat torus $\Torus_\square$.  Symmetry implies that $\Gamma$ is in equilibrium with respect to the uniform stress $\omega \equiv 1$.  Straightforward calculation gives us the covariance parameters~$\alpha = \beta = 2$ and $\gamma = 1$ for this stress vector.  Thus, Lemma~\ref{L:square-abg} immediately implies that $\omega$ is not a reciprocal stress for $\Gamma$; rather, by Lemma~\ref{L:force}, the force diagram of $\Gamma$ with respect to $\omega$  lies on the torus $\Torus_M$, where $M = \begin{psmallmatrix} \beta & -\gamma \\ -\gamma & \alpha \end{psmallmatrix} = \begin{psmallmatrix} 2 & -1 \\ -1 & 2 \end{psmallmatrix}$.  Moreover, because $\alpha\beta-\gamma^2 = 3\ne 1$, Lemma~\ref{L:nonsq-abg} implies that $\omega$ is not a reciprocal stress for the affine image of $\Gamma$ on \emph{any} flat torus.  In short, there are no reciprocal embeddings of $\Gamma$ and~$\Gamma^*$ on \emph{any} flat torus such that corresponding primal and dual edges have equal length.

Now consider the scaled uniform stress $\omega \equiv 1/\sqrt3$, which has covariance parameters~$\alpha = \beta = 2/\sqrt{3}$ and $\gamma = 1/\sqrt{3}$.  This new stress $\omega$ is still not a reciprocal stress for $\Gamma$; however, it does satisfy the scaling constraint $\alpha\beta - \gamma^2 = 1$. Lemma~\ref{L:nonsq-abg} (or Lemma \ref{L:forceM}) implies that $\omega$ is a reciprocal stress for the affine image of $\Gamma$ on $\Torus_M$, where $M = \frac{1}{\sqrt{3}} \begin{psmallmatrix} 2 & -1 \\ 0 & \sqrt{3} \end{psmallmatrix}$.  The fundamental parallelogram $\lozenge_M$ is the union of two equilateral triangles with height~$1$.  Not surprisingly, this transformed embedding is a Delaunay triangulation with equilateral triangle faces, and the faces of the reciprocal Voronoi diagram $\Gamma^*$ (which is also the force diagram) are regular hexagons.  Finally, the vector $\omega^* \equiv \sqrt{3}$ is a reciprocal stress, and therefore an equilibrium stress, for $\Gamma^*$.  See Figure~\ref{F:skew}.

\begin{figure}[ht]
\centering
\raisebox{-0.5\height}{\includegraphics[scale=0.5,page=1]{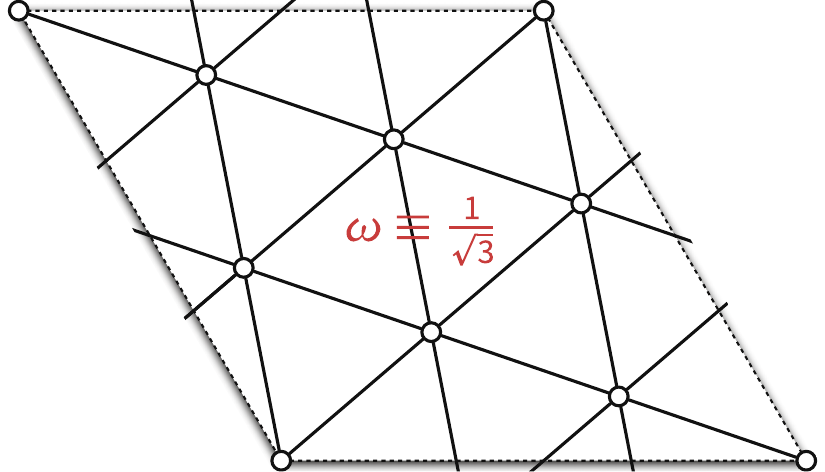}}
\raisebox{-0.5\height}{\includegraphics[scale=0.5,page=2]{Fig/K7-torus-skew}}
\caption{A seven-vertex Delaunay triangulation and its dual Voronoi diagram, induced by the uniform stress $1/\sqrt{3}$; compare with Figures \ref{F:Delaunay} and \ref{F:uniform}.}
\label{F:skew}
\end{figure}

\section{A Toroidal Steinitz Theorem}
\label{S:Steinitz}

Finally, Theorem \ref{Th:tutte} and Theorem \ref{T:skew-rec-nonlinear} immediately imply a natural generalization of Steinitz's theorem to graphs on the flat torus.

\begin{theorem}
\label{Th:torus-steinitz}
Let $\Gamma\colon G\to \Torus_\square$ be any essentially simple, essentially 3-connected embedding on the square flat torus, and let $\omega$ be \textbf{any} positive stress on the edges of $\Gamma$.  Then $\Gamma$ is homotopic to a geodesic embedding $\Gamma'\colon G\to \Torus_\square$ whose affine image on some flat torus $\Torus_M$ is coherent, such that $\omega$ is a reciprocal stress for $\underline{M}\circ \Gamma'$.
\end{theorem}

As we mentioned in the introduction, Mohar's generalization \cite{m-cpmec-97} of the Koebe-Andreev circle packing theorem already implies that each essentially simple, essentially 3-connected torus embedding~$\Gamma$ is homotopic to \emph{one} coherent embedding on \emph{one} flat torus.  In contrast, our results characterize \emph{all} coherent embeddings on \emph{all} flat tori.  Every positive stress vector $\omega\in\Real^E$ for~$\Gamma$ corresponds to a coherent embedding homotopic to $\underline{M}\circ\Gamma$, which is unique up to translation, on a flat torus $\Torus_M$, which is unique up to similarity of the fundamental parallelogram~$\lozenge_M$.  On the other hand, Lemmas~\ref{L:rec-implies-equ} and \ref{L:coh-implies-recip} imply that every coherent embedding on every flat torus corresponds to a unique positive equilibrium stress.


\paragraph*{Acknowledgements.}  We thank the anonymous reviewers of both the conference \cite{el-tmcdc-20} and journal versions of this paper for their helpful comments and suggestions.

\bibliographystyle{bib/newuser-doi}
\bibliography{mctorus-full}

\end{document}